\documentclass[12pt, leqno]{amsart}
\usepackage{amsmath}
\usepackage{amssymb}
\usepackage{amsthm}
\usepackage{enumerate}
\usepackage[mathscr]{eucal}
\usepackage{xcolor}
\theoremstyle{plain}
\usepackage{tikz}
\usepackage{soul}
\usepackage[normalem]{ulem}
\newtheorem{theorem}{Theorem}[section]
\newtheorem{lemma}[theorem]{Lemma}
\newtheorem{prop}[theorem]{Proposition}
\theoremstyle{definition}
\newtheorem{definition}[theorem]{Definition}
\newtheorem{remark}[theorem]{Remark}

\newtheorem{question}[theorem]{Question}

\newtheorem{cor}[theorem]{Corollary}
\theoremstyle{remark}

\begin{document}
	\title[Ball Covering Property of Banach spaces]{ On the Ball Covering Property of $\mathbb{L}(\mathbb{X}, \mathbb{Y})$ }
	
	\author{Ankan Mishra}
	
    	\address[Mishra]{Department of Mathematics\\ Indian Institute of Information Technology, Raichur\\ Karnataka 584135 \\INDIA. }	\email{ankanabmishra1997@gmail.com}
    	
   	\author{Kallol Paul}
    	
    	\address[Paul]{Vice-Chancellor, Kalyani University, West Bengal 741235, India and Professor (on lien)	Department of Mathematics, Jadavpur University, Kolkata 700032, India.}	\email{kalloldada@gmail.com}
	
	\author{Debmalya Sain}
	
    	\address[Sain]{Department of Mathematics\\ Indian Institute of Information Technology, Raichur\\ Karnataka 584135 \\INDIA. }	\email{saindebmalya@gmail.com}
	
	\author{Shamim Sohel}
	
		\address[Sohel]{Department of Mathematics, Indian Institue of Technology Madras, Chennai, Tamil Nadu 600036, India. } \email{shamimsohel11@gmail.com}
	
		\thanks{The research of Mr. Ankan Mishra is supported by an institute fellowship from IIIT Raichur.}  
		
		\thanks{Dr. Debmalya Sain delightfully acknowledges the support of ANRF-ECRG for the project ``Applications of the norm attainment problem in the geometry of Banach spaces and topological vector spaces"(ANRF/ECRG/2024/000436/PMS)}
		
		\thanks{ Dr. Shamim Sohel would like to thank ANRF, Govt. of India, for the financial support in the form of National Post Doctoral Fellowship (File No: PDF/2025/ 001669) under the mentorship of Dr. Surjit Kumar. }

	\begin{abstract}

We investigate the Ball Covering Property (BCP) of Banach spaces through the lens of  Birkhoff-James orthogonality, yielding a new geometric characterization of the property. Applying this framework, we characterize the BCP of the bounded linear operator space $\mathbb{L}(\mathbb{X}, \mathbb{Y})$ under specific conditions on $\mathbb{X}$ and $\mathbb{Y}$, proving that an earlier necessary condition is  sufficient.    As an application, we provide a complete affirmative answer to an open question concerning the BCP of $\mathbb{L}(L^p[0,1])$. We further apply our results to vector-valued Lipschitz spaces $\operatorname{Lip}_0(M,\mathbb{Y})$, obtaining a characterization of the BCP in this setting under the assumption of the Radon-Nikod\'{y}m property.
We also study the stability of the BCP under $p$-norm direct sums. In finite dimensions, we provide a sufficient condition for an $n$-dimensional Banach space to have a minimal covering by $n+1$  balls. Furthermore, we find an upper bound for the minimal ball covering number of $\mathbb{L}(\mathbb{X}, \mathbb{Y})$ in the finite-dimensional setting and prove that this number is exactly $mn+1$ when $\mathbb{X}$ is an $m$-dimensional strictly convex space and $\mathbb{Y}$ is an $n$-dimensional smooth space.

\end{abstract}

	\maketitle
	\section{Introduction}

    The geometry of Banach spaces continues to be an active area of research, with various geometric properties providing deep insights into the structure of normed linear spaces. Among these, the Ball Covering Property (BCP), introduced by Cheng \cite{C}, has attracted considerable attention in recent years. This property, which concerns coverings of the unit sphere by balls avoiding the origin, has revealed interesting connections with smoothness, convexity, and the overall geometry of Banach spaces. Subsequent studies have further demonstrated its relevance in Banach space geometry in general, and non-commutative spaces of operators in particular (see, for instance, \cite{CCL, CCS,CKWZ,  GLM, LLW, LZ,  Sain25, SC, SC2}). The primary objective of this article is to investigate the Ball Covering Property through the framework of Birkhoff–James orthogonality and to develop new geometric characterizations and stability results. Let us now mention the basic notations and definitions to be used in this article.\\
	
		Let $\mathbb{X}, \mathbb{Y}$ denote Banach spaces over the real field.  We use the notations $ B_\mathbb{X}$ and $S_{\mathbb{X}}$ to denote the unit ball and unit sphere of $\mathbb{X},$ respectively.  Let $\mathbb{L}(\mathbb{X}, \mathbb{Y})$ $(\mathbb{K}(\mathbb{X}, \mathbb{Y}))$ be the space of all bounded (compact) linear operators between $\mathbb{X}$ and $\mathbb{Y}.$ The space of all finite-rank operators between $\mathbb{X}$ and $\mathbb{Y}$ is denoted by $\mathbb{F}(\mathbb{X},\mathbb{Y})$. The dual space of $\mathbb{X}$ is denoted by $\mathbb{X}^*.$ For a non-zero $x \in \mathbb{X},$  $x^* \in S_{\mathbb{X}^*}$ is said to be a supporting functional at $x$ if $x^*(x)=\|x\|.$ The set of all supporting functionals at $x$ is denoted by $J(x),$ i.e., $J(x) :=\{x^* \in S_{\mathbb{X}^*}: x^*(x)=\|x\|\}.$ A non-zero  element $x \in \mathbb{X} $ is said to be smooth if $J(x)$ is a singleton. A Banach space $\mathbb{X}$ is called smooth if every $x \in S_{\mathbb{X}}$ is smooth. The set of all smooth points of $\mathbb{X}$ is denoted by $Sm(\mathbb{X}).$ We recall that the norm of $ \mathbb{X} $ is G\^ateaux differentiable at a non-zero $x \in \mathbb{X} $ if $ \lim_{t \to 0} \frac{\|x + t y\| - \|x\|}{t} $ exists for every $ y \in \mathbb{X}. $ It is well known that $ x $ is smooth if and only if the norm of $ \mathbb{X} $ is G\^ateaux differentiable at $x$ (see \cite{BS}). We also recall that the norm of $ \mathbb{X} $ is Fr\'echet differentiable at a non-zero $ x $ if $ \lim_{t \to 0} \frac{\|x + t y\| - \|x\|}{t} $ exists for every $ y \in \mathbb{X} $ and the convergence is uniform. It is clear that Fr\'echet differentiability implies G\^ateaux differentiability but the converse is true only in the finite-dimensional case. We refer to \cite{M} for further information on the topic of differentiability in Banach spaces.  An element $x \in S_{\mathbb{X}}$ is said to be an exposed point of $ B_{\mathbb{X}}$ if there exists $x^* \in J(x)$ such that $x^*(y) < 1= x^*(x),$ for any $y \in S_{\mathbb{X}} \setminus \{x\}.$ Clearly, every exposed point of $ B_{\mathbb{X}}$ is also an extreme point of $ B_{\mathbb{X}},$ but the converse is not necessarily true.  We say that $x \in S_{\mathbb{X}}$ is a strongly exposed point of $B_{\mathbb{X}}$ if there exists $x^* \in J(x)$ such that for any sequence $\{x_n\} \subset  B_{\mathbb{X}},$ $x^*(x_n) \to
		1=x^*(x)$ implies that $x_n \to x.$ In this case, we say that $x^*$ strongly exposes $x$ in $B_{\mathbb{X}}$ and $x^*$ is said to be a strongly exposing functional for $x \in B_{\mathbb{X}}$. We write $SE(B_\mathbb{X})$ for  the set of all  strongly exposing functionals of $B_{\mathbb{X}}$ and $SE(\mathbb{X}) := \{ \lambda x^*: \lambda \in \mathbb{R} \setminus\{0\}, x^* \in SE(B_{\mathbb{X}})\}.$ An element $ x^* \in S_{\mathbb{X}^*} $ is called weak*-strongly exposed point of $ B_{\mathbb{X}^*} $ if there exists an $ x \in S_{\mathbb{X}} $ such that for any sequence $ \{x_n^*\} \subset B_{\mathbb{X}^*}$ with $ x_n^*(x)\to 1 $ implies that $ x_n^* \xrightarrow{w^*} x^*.$

    We now introduce the Ball Covering Property (BCP), which is the main focus of this article. For any $ x \in \mathbb{X} $ and any $ r>0, $ $ B(x, r) $ denotes the open ball centered at $ x $ and having radius $ r. $ Similarly, $ B[x, r] $ denotes the closed ball centered at $ x $ having radius $ r. $
        
	\begin{definition}
	    A Banach space $\mathbb{X}$ is said to have the \emph{Ball Covering Property} (BCP) if its unit sphere $S_{\mathbb{X}}$ can be covered by countably many closed balls (equivalently, open balls), none of which contains the origin. In other words, there exists a sequence of balls $\{B(x_n, r_n)\}_{n=1}^\infty,$ with $r_n \leq \|x_n\|$ for each $ n\in \mathbb{N}, $ such that \[
S_{\mathbb{X}} \subseteq \bigcup_{i=1}^\infty B(x_n, r_n).
            \]
	\end{definition}	
 In case of finite-dimensional Banach spaces, one is naturally led to study the \emph{minimal ball covering number}, denoted by $\mathcal{B}_{\min}^{\#}$, which represents the smallest number of balls required to cover the unit sphere while avoiding the origin. Cheng \cite{C} established several foundational results in this direction, including the fact that for any $n$-dimensional Banach space, $n+1 \leq \mathcal{B}_{\min}^{\#} \leq 2n.$ Cheng’s work highlights the strong influence of geometric structure on ball coverings, including the role of smoothness. For an $n$-dimensional smooth Banach space, $\mathcal{B}_{\min}^{\#}= n+1,$ although the converse does not hold true in general \cite{C}. Furthermore, an $n$-dimensional Banach space is isometric to $\ell_\infty^n$ if and only if $\mathcal{B}_{\min}^{\#}= 2n. $ In the infinite-dimensional case, it is elementary to see that every separable Banach space admits the BCP. Also, it is not difficult to see that whenever a Banach space has the BCP, the centers of these balls can be chosen from any dense subset of that Banach space. Since $ Sm(\mathbb{X}) $ is dense in $ \mathbb{X} $ for every separable Banach space $\mathbb{X}, $ we can choose the center of each ball to be a smooth point, without any loss of generality. This result for separable Banach spaces was observed  by Cheng in \cite{C}. In the non-separable setting, the situation is more subtle. Examples of non-separable Banach spaces possessing the BCP include $\ell_\infty$ and $ \mathbb{L}(\mathbb{H}).$ On the other hand, the quotient spaces $\ell_\infty/ c_0$, $ \mathbb{L}(\mathbb{H}) / \mathbb{K}(\mathbb{H})$ fail to admit the BCP.  We refer the reader to \cite{BLS} for further examples and related results in this direction.  The BCP is also linked with several geometric properties like Radon-Nikod\'{y}m property \cite{CZZ}, strict convexity and dentability \cite{SC, SC2} and uniform non-squareness \cite{CCS, CLZZ}. \\

The ball covering property has been studied from several perspectives in Banach space theory. In this article, we establish a connection between the ball covering property and Birkhoff-James orthogonality, leading to a complete characterization in terms of the orthogonality relation.
We recall from  \cite{B, J} that  an element $x \in \mathbb{X}$ is said to be Birkhoff-James orthogonal to $y \in \mathbb{X}$ if 
$$ \| x + \lambda y \| \geq \|x\|, \, \, \text{for all } \lambda \in \mathbb{R}.$$ 
Symbolically, it is written as $ x \perp_B y.$ Given any $x \in \mathbb{X},$  we write $x^\perp := \{ y \in X: x \perp_B y\}.$ For real Banach spaces, this orthogonality relation is  dissected in two parts as $x^+$ and $x^-,$ in \cite{S2}. Given $x, y \in \mathbb{X},$ we say that $y \in x^+$   if $\|x+ \lambda y\| \geq \|x\|$ for any $\lambda \geq 0$ and that $y \in x^-$ if $\|x+ \lambda y\| \geq \|x\|$ for any $\lambda \leq 0.$ For any $x, y \in \mathbb{X},$ it is easy to verify that either $y \in x^+$ or $y \in x^-.$  We note that $x \perp_B y \iff y \in x^+ \cap x^-.$
The connection between these notions and the ball covering property is given by the following characterization, which will serve as the principal tool throughout the paper. It is also the first main result of the present work.

\begin{lemma}
    Let $\mathbb{X}$ be a Banach space. Then $\mathbb{X}$ has the BCP if and only if there exists a countable set $\{x_i: i \in \mathbb{N}\} \subset \mathbb{X} \setminus \{0\}$ such that $\cap_{i \in \mathbb{N}} x_i^\perp = \{0\}.$
\end{lemma}
   
With this characterization in hand, we turn to the study of the BCP in spaces of bounded linear operators. The BCP has recently been investigated in both commutative function spaces and non-commutative operator spaces in  \cite{LLLZ}. In that work, a necessary condition for the BCP of $\mathbb{L}(\mathbb X,\mathbb Y)$ were obtained for general Banach spaces $\mathbb{X}$ and $\mathbb{Y}$. 
In the present work, we continue this line of investigation and obtain a characterization of the BCP for subspaces of $\mathbb L(\mathbb X,\mathbb Y)$ under a natural density assumption.

 \begin{theorem}
          Let $\mathbb{X}, \mathbb{Y}$ be  Banach spaces such that either of the following conditions holds true: 
          \begin{itemize}
              \item [$(i)$] $SE(\mathbb{X})$ is dense in $\mathbb{X}^*$

              \item[$(ii)$] The set of all Fr\'echet differentiable points is dense in $\mathbb{Y}.$ 
          \end{itemize}
          Let $\mathcal{W}$ be a subspace of $\mathbb{L}(\mathbb{X}, \mathbb{Y})$ containing $\mathbb{F}(\mathbb{X}, \mathbb{Y}).$ Then  $\mathcal{W}$ has the BCP if and only if both $\mathbb{X}^*$ and $\mathbb{Y}$ have the BCP.
    \end{theorem}
 
 
 The above result is established in  Theorem~\ref{theorem:characterization} as an important application of this study carried out by us.  It yields a complete affirmative answer to an open question posed in \cite{BLS} concerning the BCP of $\mathbb{L}(L^p[0,1])$. Moreover, it enables us to refine several corresponding results on the BCP of non-commutative operator spaces established in \cite{LLLZ}.

 As a further application, we consider the BCP in Lipschitz spaces, $\operatorname{Lip}_0(M,\mathbb Y)$, the space of all Lipschitz functions from a pointed metric space $M$ to a Banach space $\mathbb Y$, equipped with the Lipschitz norm. The natural isometric isomorphism between $\operatorname{Lip}_0(M,\mathbb Y)$ and $\mathbb L(\mathcal F(M),\mathbb Y),$ where $\mathcal{F}(M)$ is the Lipschtiz-free space,  provides a  route to the study of the BCP in this setting. As a consequence, we obtain the following result (see Corollary~\ref{cor:lip}).

\begin{theorem}
Let $M$ be a pointed metric space and let $\mathbb Y$ be a Banach space such that either $\mathcal F(M)$ or $\mathbb Y^*$ has the RNP. Then $\operatorname{Lip}_0(M,\mathbb Y)$ has the BCP if and only if both $\operatorname{Lip}_0(M)$ and $\mathbb Y$ have the BCP.
\end{theorem}

We further investigate the BCP in the scalar-valued Lipschitz space $\operatorname{Lip}_0(M)$, obtaining sufficient conditions for the BCP as well as a characterization under the additional assumption that $\mathcal F(M)$ has the RNP.

 We also investigate the stability of the BCP under countable direct sums equipped with the $p$-norm. In the final  section we study  the problem of minimal ball covering number for the case of finite-dimensional Banach spaces. We obtain a tractable characterization of the minimal ball covering number of a finite-dimensional Banach space in terms of Birkhoff-James orthogonality. 
We also provide a sufficient condition for an $n$-dimensional Banach space to have the minimal ball covering number $n+1$, which covers both smooth Banach spaces of dimension $ n $ and $ \ell_{1}^n. $ Let us mention here that the minimal ball covering number of both these spaces were known previously (\cite{C}, \cite{CLZZ}) but the proofs varied significantly. Our main contribution in this regard is to show that by using the language of Birkhoff-James orthogonality, it is possible to unify the proofs and to generalize the result for larger classes of Banach spaces. As a further consequence of our approach involving orthogonality, we obtain upper bounds for the minimal ball covering numbers in spaces of bounded linear operators between finite-dimensional Banach spaces. In particular, for strictly convex and smooth Banach spaces, we obtain the following result (see Theorem~\ref{theorem:same number}).

\begin{theorem}
	Let $ \mathbb{X}$ be an $m$-dimensional strictly convex Banach space and let $\mathbb{Y}$ be an $n$-dimensional smooth Banach space. Then $S_{\mathbb{L}(\mathbb{X}, \mathbb{Y})}$ has a ball covering consisting of $mn+1$ balls.
\end{theorem}

\section{Preliminaries}

The purpose of this section is to recall some important results in the theory of ball covering properties of Banach spaces, mainly for the convenience of the readers. These results will be used frequently throughout this article without further explanation. Let us begin with the following fundamental characterization of Birkhoff-James orthogonality.	
	
	\begin{theorem}\cite{J}\label{J}
		Let $\mathbb{X}$ be a  Banach space and let $ x, y \in \mathbb{X}.$ Then $x \perp_B y$ if and only if there exists $f \in \mathbb{X}^*$ such that $f(x)=\|f\| \|x\|$ and $ f(y)=0. $
	\end{theorem}

In real Banach spaces, Birkhoff-James orthogonality relation essentially consists of two parts. It is trivial to see that for any $ x, y \in \mathbb{X}, $ $ x \perp_B y $ if and only if $ y \in x^{+} \cap x^{-}. $ We recall the following important functional characterizations of the positive and negative parts of an element in a Banach space.
    
	\begin{lemma}\label{functional}\cite[Lemma. 7.3.2]{MPS}
		Let $\mathbb{X}$ be a real Banach space and $x, u \in \mathbb{X}.$ Then the following holds:
		\begin{itemize}
        \item[$(i)$] $u \in x^-$ if and only if $ f(u) \leq 0$ for all $ f \in J(x).$
        \item[$(ii)$] $u \in x^- \setminus x^+$ if and only if $f(u) < 0 $ for all $ f \in J(x)$.
        \item[$(iii)$] $u \in x^+$ if and only if $f(u) \geq 0 $ for all $ f \in J(x).$
			\item[$(iv)$] $u \in x^+ \setminus x^-$ if and only if $f(u) > 0 $ for all $ f \in J(x).$
					\end{itemize}
	\end{lemma}


Recently in \cite{Sain25}, a local version of the BCP has been studied using the notion of norm derivatives. We recall the result below \cite{Sain25}.

 \begin{theorem}\cite [Th. 2.1]{Sain25}
  Let $\mathbb{X}$ be a Banach space and let $x, y \in \mathbb{X}$ be non-zero elements. Then the following conditions are equivalent:
\begin{itemize}
    \item[$(i)$] There exists a ball centered at $\lambda x $, for some $\lambda > 0$, which contains $y$ but does not contain the zero vector,
    
    \item[$(ii)$] $\rho'_{-}(x, y) > 0$.
\end{itemize}
\end{theorem}

The following result from \cite{Sain25} is useful in our study, particularly in the finite-dimensional case. 

\begin{prop}\cite[Prop. 2.4]{Sain25}\label{prop:compact}
	Let $\mathbb{X}$ be a Banach space and let $ A \subset \mathbb{X}$ be such that $A \subset \cup_{i=1}^n
	B(\lambda_ix, r_i),$ where $n \in \mathbb{N}$ and
	$x \in S_{\mathbb{X}}$ are fixed, $\lambda_i > 0,$ and $0 < r_i \leq \lambda_i$ for each $1 \leq  i \leq n.$ Then for each $\lambda \geq  \max \{\lambda_i : 1 \leq i \leq n\}$, there
	exists $0 < r_\lambda \leq \lambda$ such that $A \subset B(\lambda x, r_{\lambda}).$
\end{prop}

   To facilitate the arguments in the subsequent sections, we require certain structural properties of strongly exposing functionals, Fr\'echet differentiability, and annihilator subspaces, which will be used throughout the paper. The following lemma describes the dual behavior of a strongly exposing functional. Although elementary, it plays an essential role in the study of orthogonality relations in dual Banach spaces.
	
	\begin{lemma}\label{lemma:SE}
   Let $x^* \in SE(\mathbb{X})$ and let $z^*$ strongly exposes $z \in S_{\mathbb{X}}.$ Then 
   \begin{itemize}
       \item [$(i)$] $z^*$ is smooth and $J(z^*)=\{\psi(z)\}$
       \item[$(ii)$] $(z^*)^\perp = \{ x^*\in \mathbb{X^*}: x^*(z)=0\},$ 
        \end{itemize}
         where $\psi $ is the canonical isometric isomorphism from $\mathbb{X}$ to $\mathbb{X}^{**}.$
\end{lemma}

\begin{proof}
    $(i)$ Let $\phi \in J(z^*).$ By Goldstine’s Theorem, the unit ball $B_{\mathbb{X}}$ is weak*-dense in $B_{\mathbb{X}^{**}}.$ This implies that there exists a net $\{z_\alpha\} \subset B_{\mathbb{X}}$ such that $\psi(z_\alpha) \xrightarrow{w^*} \phi$. Therefore, $\psi(z_\alpha)(z^*) \to \phi(z^*) =\|z^*\|,$ which implies that $z^*(z_\alpha) \to z^*(z).$ As $z$ is a strongly exposed point of $B_{\mathbb{X}},$ we get that $z_\alpha \to z$ and therefore, $\psi(z_\alpha) \to \psi(z).$ Hence $\phi = \psi(z).$ This shows that $z^*$ is smooth and $J(z^*)=\{ \psi(z)\}.$

    $(ii)$ Let $x^* \in \mathbb{X}^*$ be such that $z^* \perp x^*.$ Since $J(z^*)= \psi(z),$ it follows from Theorem \ref{J} that $\psi(z)(x^*)=0 \implies x^*(z)=0.$ Therefore, $(z^*)^\perp = \{ x^*\in \mathbb{X^*}: x^*(z)=0\}.$ 
\end{proof}

   The following lemma characterizes Fréchet differentiability of the norm in terms of supporting functionals.
	
	\begin{lemma}\label{lemma:Frechet differentiable}\cite[Prop. 5.11]{P}
  Let $x_0 \in S_{\mathbb{X}}$. The norm of $\mathbb{X}$ is Fr\'echet differentiable at $x_0$ if and only if there exists a unique $x_0^* \in S_{\mathbb{X}^*}$ such that $x_0^*(x_0)=1$, and $x_0^*$ is weak*-strongly exposed by $x_0$.
	\end{lemma}

    We conclude this section by recalling some well-known properties of annihilators that will be used later. To avoid notational confusion with the left and the right annihilators, we reserve the symbol $ \perp $ for Birkhoff-James orthogonality. For any subspace $ \mathbb{Y} $ of $ \mathbb{X}, $ $ R_{Ann}(\mathbb{Y}):= \{ x^* \in \mathbb{X^*}: x^*(x)=0, \forall x\in \mathbb{Y}\} $ denotes the right-annihilator of $ \mathbb{Y}. $ Similarly for any subspace $ \mathbb{Z} $ of $ \mathbb{X^*}, $ $ L_{Ann}(\mathbb{Z}):= \{x\in \mathbb{X}:x^*(x)=0, \forall x^* \in \mathbb{Z}\} $ denotes the left-annihilator of $ \mathbb{Z}. $ Clearly, $ R_{Ann}(\mathbb{Y}) $ and $ L_{Ann}(\mathbb{Z}) $ are subspaces of $ \mathbb{X^*} $ and $ \mathbb{X},$ respectively.
    
\begin{lemma}\label{Rudin}\cite[Th. 4.7]{R}
Let $\mathbb{Y}$ be a subspace of $\mathbb{X}$ and let $\mathbb{Z}$ be a subspace of $\mathbb{X}^*.$ Then
   \begin{itemize}
       \item [$(i)$] $ L_{Ann}( R_{Ann} (\mathbb{Y}))= \overline{\mathbb{Y}}. $
       \item [$(ii)$] $R_{Ann} ( L_{Ann} (\mathbb{Z}))= \overline{\mathbb{Z}}^{w^*}. $
       \item[$(iii)$]  $\mathbb{Y}^*= \mathbb{X}^* / \overline{\mathbb{Y}}^{w*}.$ 
       
   \end{itemize}
\end{lemma}

\section{Main Results}

\section*{Section-I}

\section*{Ball covering property of $\mathbb{L}(\mathbb{X}, \mathbb{Y}).$}

We begin by noting that in \cite[Th. 3.2]{LLLZ}, a necessary condition was obtained for the BCP of the operator space $\mathbb{L}(\mathbb{X}, \mathbb{Y})$ in terms of the BCP of $ \mathbb{X}^{*} $ and $ \mathbb{Y}. $
\begin{theorem} \cite[Th. 3.2]{LLLZ}\label{JFA}
		Let $\mathbb{X}$ and $\mathbb{Y}$ be Banach spaces. If $\mathbb{L}(\mathbb{X}, \mathbb{Y})$ has the BCP, then both
		$\mathbb{X}^*$ and $\mathbb{Y}$ have the BCP.
	\end{theorem}

Motivated by the above result, our goal is to investigate the converse of the same. The following lemma, that characterizes the BCP of a Banach space in terms of Birkhoff-James orthogonality properties, is important for our purpose. 
\begin{lemma}\label{lemma:main}
 	
 	Let $ \mathbb{X} $ be a Banach space. Then the following are equivalent:
 	
 	\begin{itemize}
 		\item[$(i)$] $\mathbb{X}$ has the BCP,
 		
 		\item[$(ii)$] There exists a countable set $\{ x_i : i \in \mathbb{N}\} \subset \mathbb{X}\setminus\{0\} $ such that $\cap_{i \in \mathbb{N}}x_i^- = \{0\}, $
 		
 		\item[$(iii)$] There exists a countable set $  \{ y_i : i \in \mathbb{N}\} \subset \mathbb{X}\setminus\{0\} $ such that $\cap_{i \in \mathbb{N}}y_i^+ = \{0\}, $
 		
 		\item[$(iv)$] There exists a countable set $  \{ z_i : i \in \mathbb{N}\} \subset \mathbb{X} \setminus \{0\} $ such that $\cap_{i \in \mathbb{N}}z_i^\perp = \{0\} .$

 	\end{itemize}
 	
 \end{lemma}
 
 \begin{proof}
 	$ (i) \implies (ii): $ Since $\mathbb{X}$ has the BCP, there exists a countable set $\{ x_i : i \in \mathbb{N}\}\subset \mathbb{X} \setminus \{0\} $ such that $ S_{\mathbb{X}} \subset \cup_{i \in \mathbb{N}} B(x_i, r_{x_i}), $ where $ 0< r_{x_i}\leq \|x_i\| $ for each $ i \in \mathbb{N}. $ So, for any $z \in S_{\mathbb{X}},$ there exists $ j \in \mathbb{N}$ such that $z\in B(x_j, r_{x_j}). $ This implies that $ \| x_j -z \|< \|x_j \|, $ which is equivalent to $ z \not\in x_j^-. $ Therefore, $ z \not\in \cap_{i\in \mathbb{N}}x_i^-. $ Since $ z \in S_\mathbb{X} $ is arbitrary, it follows that $ S_{\mathbb{X}} \bigcap (\cap_{i\in \mathbb{N}}x_i^-) = \phi.$ Now using the homogeneity of the Birkhoff-James orthogonality, it is easy to verify that $ \cap_{i\in \mathbb{N}} x_i^- = \{0\}.$  \\
 	
 	
 	 $(ii) \implies (iii):$ Follows trivially by choosing $ y_i = - x_i, $ for each $ i \in \mathbb{N}. $\\
 	
 	$ (iii) \implies (iv):$ Follows immediately from the fact that for any $x \in \mathbb{X},$ $x^\perp \subset x^{+}.$\\
 	 
 	 $ (iv) \implies (i): $ Let $y \in S_{\mathbb{X}}. $ As $\cap_{i \in \mathbb{N}} z_i^\perp = \{0\}, $ there exists $ k \in \mathbb{N}$ such that $ y \notin z_k^\perp. $ This implies that there exists a scalar $\lambda_y $ such that $\| z_k + \lambda_y  y \| < \|z_k\|. $ Suppose that $ \lambda_y < 0. $ Let $ \lambda_y= -\frac{1}{t_y} $ for some $ t_y>0. $  This implies that $ \|t_y z_k - y \| < \|t_y z_k\|. $ Choose any $ n_0 \in \mathbb{N}$ with $ n_0 \geq t_y. $ It is easy to verify that $ y \in B(n_0z_k, n_0 \|z_k\|).$ Similarly, whenever $\lambda_y > 0, $ there exists $n_1 \in \mathbb{N}$ such that $ y \in B(n_1(-z_k), n_1 \|z_k\|).$
    Now, consider the countable collection of balls,
 	 \[
 	 \mathcal{B'}= \bigg\{ B(n z_i, n\|z
 	 _i\|):  i, n \in \mathbb{N} \bigg\} \bigcup  \bigg\{ B(-n z_i, n \|z_i\|): i, n \in \mathbb{N} \bigg\}.
 	 \]
 	Clearly, $ y \in \mathcal{B'}. $ Since $ y \in S_{\mathbb{X}} $ is arbitrary, it follows that $ S_{\mathbb{X}} \subset \mathcal{B'}. $ This completes the lemma.
 	
 \end{proof}

	

  \begin{remark}
In the reflexive and smooth setting, the characterization of the BCP admits a natural dual reformulation. In this case, $\mathbb{X}$ has the BCP if and only if there exists a countable family $\{f_i\}_{i\in\mathbb N} \subset \mathbb{X}^*$ such that $\overline{\mathrm{span}\{f_i : i\in\mathbb N\}}^{\,w^*} = \mathbb{X}^*$. Observe that if $\mathbb{Z} = \mathrm{span}\{f_i : i\in\mathbb N\}$, then by Lemma~\ref{Rudin}, we have $L_{Ann}(\mathbb{Z}) = \bigcap_{i\in\mathbb N} \ker f_i$ and $R_{Ann}(L_{Ann}(\mathbb{Z})) = \overline{\mathbb{Z}}^{\,w^*}$. Consequently, $ \mathbb{X} $ has the BCP implies that $\overline{\mathbb{Z}}^{\,w^*} = \mathbb{X}^*. $ On the other hand, $\overline{\mathbb{Z}}^{\,w^*} = \mathbb{X}^* $ implies that  $\bigcap_{i\in\mathbb N} \ker f_i = \{0\}$. Since $\mathbb{X}$ is smooth, for each $f_i$ there exists a unique $x_i \in \mathbb{X}$ such that $J(x_i)=\{f_i\}$, and in this case $x_i^\perp = \ker f_i$. Consequently, $\bigcap_{i\in\mathbb N} x_i^\perp = \{0\}$, which is precisely the characterization of the BCP given in Lemma~\ref{lemma:main}. 
\end{remark}


  As mentioned earlier, our objective is to investigate the opposite direction of Theorem \ref{JFA}. For this purpose, we require the notion of absolutely strongly exposing operators. An operator $T \in \mathbb{L}(\mathbb{X}, \mathbb{Y})$ is said to be an absolutely strongly exposing operator (see \cite{Jung} for more information) if there exists $x_0 \in B_{\mathbb{X}}$ such that whenever a sequence $\{x_n\} \subset B_{\mathbb{X}}$ satisfies $\|Tx_n\| \to \|T\|,$ then there exists a sequence $\{\theta_n\} \subset \{ \pm 1 \}$ such that $\theta_n x_n \to x_0.$  It is trivial to see that for such an operator $T,$ $M_T = \{\pm x_0\}.$ 
    
The set of absolutely strongly exposing operators of $\mathbb{L}(\mathbb{X}, \mathbb{Y})$ is denoted as $ASE(\mathbb{X}, \mathbb{Y})$. The following lemma shows that for absolutely strongly exposing operators, Birkhoff-James orthogonality can be characterized in a particularly nice form.

	 \begin{lemma}\label{BSP2}\cite[Th. 3.33]{SGSP}
Let $\mathbb{X}$ and $\mathbb{Y}$ be Banach spaces. Suppose that $T \in ASE(\mathbb{X}, \mathbb{Y})$ with $M_T = \{\mu x_0: |\mu|=1\}.$  Then for any $A \in S_{\mathbb{L}(\mathbb{X}, \mathbb{Y})}$, 
		$T \perp_B A$ if and only if 
		$Tx_0 \perp_B Ax_0.$
		
	\end{lemma}

We are now in a position to state the main theorem of this article.

    \begin{theorem}\label{operator:main}
Let $\mathbb{X}, \mathbb{Y}$ be Banach spaces and let $\mathcal{W}$ be a subspace of $\mathbb{L}(\mathbb{X}, \mathbb{Y})$ containing $\mathbb{F}(\mathbb{X}, \mathbb{Y}).$ Then  $\mathcal{W}$ has the BCP if  one of the following holds: 
\begin{itemize}
    \item[$(i)$]  $\mathbb{Y}$ has the BCP and there exists a countable set of strongly exposed points of $B_{\mathbb{X}}$ separating $\mathbb{X}^*.$

    \item[$(ii)$]  $\mathbb{X}^*$ has the BCP and there exists a countable set of weak*-strongly exposed points of $B_{\mathbb{Y}^*}$ separating $\mathbb{Y}^{**}.$
\end{itemize}
    \end{theorem}

	\begin{proof}
    
        Suppose that $(i)$ holds true. Let $\mathcal{W} $ be a subspace of $\mathbb{L}(\mathbb{X}, \mathbb{Y})$ containing $\mathbb{F}(\mathbb{X}, \mathbb{Y}).$
		   Let us assume that the set $\{ x_i: i \in \mathbb{N}\} $ separates $\mathbb{X}^*, $ where each $x_i $ is a strongly exposed point of $B_{\mathbb{X}}.$  Suppose that for each $i \in \mathbb{N}, f_i \in S_{\mathbb{X}^*}$ is the corresponding strongly exposing functional.
		    Since $\mathbb{Y}$ has the BCP, it follows from Lemma \ref{lemma:main} that there exists a countable set $\{y_i: i \in \mathbb{N}\}$ such that $\cap_{i \in \mathbb{N}} y_i^\perp = \{0\}.$ Without any loss of generality, we may and do assume that each $y_i \in S_{\mathbb{Y}}.$
		    For any $i, j \in \mathbb{N}$, consider mappings $ T_{ij}: \mathbb{X} \to \mathbb{Y}$ given by 
		    	\begin{eqnarray*}
		    	T_{ij} (x) &:=& f_i(x)  y_j, \quad \forall x \in \mathbb{X}.
		    \end{eqnarray*}
            
		    Clearly, each $T_{ij} $ is linear and $\|T_{ij}\|=\|f_i\|=1$. Moreover, $T_{ij} \in \mathbb{F}(\mathbb{X},\mathbb{Y}) \subset \mathcal{W}.$ From Lemma \ref{lemma:main}, it is clear that to establish the BCP of  $\mathcal{W},$  it suffices to show that $\cap_{i,j \in \mathbb{N}} T_{ij}^\perp = \{0\}.$ Let $A \in \cap_{i,j \in \mathbb{N}} T_{ij}^{\perp}$ be arbitrary but fixed after choice.
		    
		    Observe that for a sequence $\{u_n\} \subset S_{\mathbb{X}}$, 
		     $$\|T_{ij} u_n\| \to \|T_{ij}\|=1 \implies |f_i(u_n)| \| y_j\| \to 1 \implies |f_i(u_n) | \to 1.$$
		 Passing onto a suitable subsequence, if necessary, let us assume without any loss of generality that $f_i(u_n) \to \pm 1.$
		  As $f_i$ is the strongly exposing functional of a strongly exposed point $x_i,$ it is easy to see that $T_{ij} \in ASE(\mathbb{X}, \mathbb{Y})$ and $M_{T_{ij}}= \{ \pm x_i\}.$ Since $T_{ij} \perp_B A,$ it follows from Lemma \ref{BSP2} that
		       $$T_{ij} x_i \perp_B Ax_i \implies y_j \perp_B Ax_i.$$ 
		       For a fixed $i \in \mathbb{N},$ we get $ Ax_i \in y_j^\perp,$ for any $j \in \mathbb{N}.$ As $\cap_{j \in \mathbb{N}} y_j^\perp =\{0\},$ we get $Ax_i=0$ for each $i \in \mathbb{N}.$ As the set $\{x_i : i \in \mathbb{N}\}$ separates $\mathbb{X}^*$ and $Ax_i=0 $ for all $ i \in \mathbb{N},$ we obtain $A=0.$ Now by applying Lemma \ref{lemma:main}, we deduce that $ (i) $ holds true.\\

		Next suppose that $(ii)$ holds true.  
		We assume that the set $\{ y_i^*: i \in \mathbb{N}\}$ separates $\mathbb{Y}^{**}$, where each $y_i^* $ is a weak*-strongly exposed point of $B_{\mathbb{X}}.$  Suppose that for each $i \in \mathbb{N}, \psi(y_i) \in S_{\mathbb{Y}^{**}}$ be such that $\psi(y_i)(y_i^*)= \|y_i^*\|=1,$ where $\psi$ is the canonical isometric isomorphism from $\mathbb{Y}$ to $\mathbb{Y}^{**}.$
		Since $\mathbb{X}^*$ has the BCP, it follows from Lemma \ref{lemma:main} that there exists a countable set $\{x_i^*: i \in \mathbb{N}\}$ such that $\cap_{i \in \mathbb{N}} (x_i^*)^\perp = \{0\}.$ Without loss of generality we assume each $x_i^* \in S_{\mathbb{X}^*}.$
			For any $i, j \in \mathbb{N}$, consider mappings $A_{ij} : \mathbb{Y}^* \longrightarrow \mathbb{X}^* $ given by 
		\begin{eqnarray*}
			A_{ij} (y^*) &=& \psi(y_i)(y^*)  x_j^*, \quad  \forall y^* \in \mathbb{Y}^*.
		\end{eqnarray*}
		Clearly, $A_{ij} $ is linear and $\|A_{ij}\|=\|\psi(y_i) \|=1.$ Following similar arguments as above, we can show that $\cap_{i,j \in \mathbb{N}} A_{ij}^\perp = \{0\}.$
		
		 Since $\psi(y_i)$ is weak*-continuous, it is easy to observe that $A_{ij}$ is  weak*-to-weak* continuous, for each $i,j \in \mathbb{N}$. Therefore, there exists  $T_{ij} \in \mathbb{L}(\mathbb{X}, \mathbb{Y})$ such that $T_{ij}^*= A_{ij} $ (see, for example,  \cite[Th. 3.1.11]{M} for details).  Clearly, $A_{ij} \in \mathbb{F}(\mathbb{X}, \mathbb{Y}) \subset \mathcal{W}.$
			Since for any $S \in \mathbb{L}(\mathbb{X}, \mathbb{Y}),$ $$T_{ij} \perp_B S \implies T_{ij}^* \perp_B S^* \implies A_{ij} \perp_B S^*$$ 
	and $\cap_{i,j \in \mathbb{N}} A_{ij}^\perp = \{0\},   $	it is easy to observe that $ \cap_{i \in \mathbb{N}} T_{ij}^\perp = \{0\}.$ It now follows from Lemma \ref{lemma:main} that $\mathcal{W}$ has the BCP.
	
	\end{proof}

    Since Theorem \ref{JFA} provides a necessary condition for the BCP of $ \mathbb{L}(\mathbb{X}, \mathbb{Y}), $ it is natural to ask how restrictive the sufficient condition obtained in Theorem \ref{operator:main} actually is. In this context, we record the following observation. 
    \begin{remark}\label{remark1}
        Following Theorem \ref{JFA} and \ref{operator:main}, we get that if there exists a countable set of strongly exposed points of $B_{\mathbb{X}}$ separating $\mathbb{X}^*,$ then $\mathbb{X}^*$ has the BCP. Similarly, if there exists a countable set of weak*-strongly exposed points of $B_{\mathbb{Y}^*}$ separating $\mathbb{Y}^{**},$ then $\mathbb{Y}$ has  the BCP. Therefore, the geometric condition used in Theorem 3.5 to ensure that the BCP of $\mathbb{L}(\mathbb{X}, \mathbb{Y})$  already guarantees the BCP of $\mathbb{X}^* $ and $\mathbb{Y}.$
 However, the converse is not true in general. Indeed, $\mathbb{X} = c_0$ 
 provides a concrete counterexample: $\mathbb{X}^* = \ell^1$ has the BCP, 
  yet $B_{c_0}$ has no strongly exposed points 
 whatsoever. Thus, the BCP of $\mathbb{X}^*$ alone does not guarantee the 
 existence of even a single strongly exposed point of $B_{\mathbb{X}},$ let 
 alone a countable separating family. This gap between the necessary condition 
 of Theorem \ref{JFA} and the sufficient condition of Theorem 
 \ref{operator:main} motivates the question of identifying  additional 
 assumptions under which the two conditions become equivalent.
 \end{remark}
 
 This naturally leads to the question of when such an equivalence can be 
 recovered. The following two results show that under a density assumption 
 on strongly exposing functionals (respectively, on Fr\'echet differentiable 
 points), the BCP of $\mathbb{X}^*$ (respectively, of $\mathbb{Y}$) is indeed 
 equivalent to the existence of the required countable separating family.

    \begin{prop}\label{Proposition: SE(X)}
         Let $\mathbb{X}$ be a Banach space such that $SE(\mathbb{X})$ is dense in $\mathbb{X}^*.$ Then the following are equivalent: 
         \begin{itemize}
             \item [$(i)$] There exists a countable set of strongly exposed points of $B_{\mathbb{X}}$ separating $\mathbb{X}^*.$
             
             \item[$(ii)$]  $\mathbb{X}^*$ has the BCP.
         \end{itemize}
    \end{prop}

    \begin{proof}
        We note that the implication $ (i) \implies (ii) $ follows from Theorem \ref{operator:main}. Accordingly, we will only prove that $(ii) \implies (i).$ It follows from Lemma \ref{lemma:main} that there exists a countable set $\{x_i^*: i \in \mathbb{N}\} \subset S_{\mathbb{X}^*}$ such that $$\bigcap_{i \in \mathbb{N}} (x_i^*)^\perp=\{0\}. $$ Observe that these $x_i^* $ are actually the centers of the balls covering $ S_{\mathbb{X}^*}. $ Since $SE(\mathbb{X})$ is dense in $ \mathbb{X}^*, $ without loss of generality we may and do assume that each $x_i^* \in SE(\mathbb{X})$. For each $i \in \mathbb{N}$, let $x_i^*$ strongly exposes $x_i \in S_\mathbb{X}.$ To complete the proof, we need to show that $\{x_i: i \in \mathbb{N}\}$ separates $\mathbb{X}^*.$ Suppose that for each $i \in \mathbb{N},$ $x^* (x_i)=0,$ for some $x^* \in \mathbb{X}^*.$  Following Lemma \ref{lemma:SE}, we get that $x^* \in (x_i^*)^\perp,$ for each $i \in \mathbb{N}.$ Hence $x^* \in \cap_{i \in \mathbb{N}} (x_i^*)^\perp= \{0\},$ which implies that $x^*=0.$ Therefore, $\{x_i: i \in \mathbb{N}\}$ separates $\mathbb{X}^*.$  
    \end{proof}

In a similar spirit to the previous observation, our next result characterizes the BCP of a Banach space under the density assumption on Fr\'echet differentiable points.

    \begin{prop}\label{Proposition: Frechet diff}
         Let $\mathbb{Y}$ be a Banach space such that the set of all Fr\'echet differentiable points is dense in $\mathbb{Y}.$ Then the followings are equivalent: 
         \begin{itemize}
             \item [$(i)$] There exists a countable set of weak*-strongly exposed points of $B_{\mathbb{Y}^*}$ separating $\mathbb{Y}^{**}, $
             
             \item[$(ii)$]  $\mathbb{Y}$ has the BCP.
         \end{itemize}
    \end{prop}

     \begin{proof}
     To prove that $  (i) \implies (ii), $ let us assume that $ \{ f_i : i \in \mathbb{N} \} $ is a countable set of weak*-strongly exposed points of $B_{\mathbb{Y}^*}$ separating $\mathbb{Y}^{**}. $ Let $ \psi(y_i) $ be the weak*-strongly exposing functional corresponding to $ f_i. $ Using the fact that $ \psi $ is an isometry, it is not hard to deduce that $ \bigcap_{i \in \mathbb{N}} (y_i)^\perp=\{0\}. $ This finishes the proof, by virtue of Lemma \ref{lemma:main}.\\
     Let us now prove that $(ii) \implies (i).$ We assume that $\mathbb{Y}$ has the BCP. It follows from Lemma \ref{lemma:main} that there exists a countable set $\{y_i: i \in \mathbb{N}\} \subset S_{\mathbb{Y}}$ such that $$\bigcap_{i \in \mathbb{N}} (y_i)^\perp=\{0\}. $$ Observe that these $y_i$ are actually the centers of the balls covering $ S_{\mathbb{Y}}. $ Since
        the set of all Fr\'echet differentiable points is dense in $\mathbb{Y},$
         the center of the balls $y_i$ can be taken as Fr\'echet differentiable points. For each $i \in \mathbb{N}$, let $J(y_i)= \{y_i^*\}.$ Following Lemma \ref{lemma:Frechet differentiable},  each $y_i^*$ is a weak*-strongly exposed points $B_{\mathbb{Y}^*}$. 
          To complete the proof, we need to show that  $\{y_i^*: i \in \mathbb{N}\}$ separates $\mathbb{Y}^{**}.$ Let $\phi \in \mathbb{Y}^{**}$ such that $\phi(y_i^*)=0, $ for all $ i \in \mathbb{N}.$ Since $\cap_{i \in \mathbb{N}} (y_i)^\perp= \{0\},$ it follows from Theorem \ref{J} that $\cap_{i \in \mathbb{N}} \ker y_i^* = \{0\}.$ Applying Lemma \ref{Rudin}, $ R_{Ann}(\cap_{i \in \mathbb{N}} \ker y_i^* ) = \overline{span\{ y_i^*: i \in \mathbb{N}\}}^{w*} \implies \overline{span\{ y_i^*: i \in \mathbb{N}\}}^{w*}= \mathbb{Y}^{*}.$ 
        Since $\phi(y_i^*)=0 $ for each $ i \in \mathbb{N}$, it is immediate that $\phi=0.$
    \end{proof}

    \begin{remark}
        The implication $ (i) \implies (ii) $ in Proposition \ref{Proposition: Frechet diff} do not require the Fr\'echet differentiable points to be dense in $\mathbb{Y}.$
    \end{remark}

As an immediate consequence of the above results, the necessary condition appearing in Theorem \ref{JFA} also becomes sufficient under this density assumption, thus producing the following characterization of the BCP of $\mathbb{L}(\mathbb{X}, \mathbb{Y})$ in this setting. 

    \begin{theorem}\label{theorem:characterization}
          Let $\mathbb{X}, \mathbb{Y}$ be  Banach spaces such that either of the following conditions holds true: 
          \begin{itemize}
              \item [$(i)$] $SE(\mathbb{X})$ is dense in $\mathbb{X}^*$

              \item[$(ii)$] The set of all Fr\'echet differentiable points is dense in $\mathbb{Y}.$ 
          \end{itemize}
          Let $\mathcal{W}$ be a subspace of $\mathbb{L}(\mathbb{X}, \mathbb{Y})$ containing $\mathbb{F}(\mathbb{X}, \mathbb{Y}).$ Then  $\mathcal{W}$ has the BCP if and only if both $\mathbb{X}^*$ and $\mathbb{Y}$ have the BCP.
    \end{theorem}

    We now recall two standard density results that provide useful sufficient conditions for the hypotheses of the above theorem.  If $\mathbb X$ has the Radon-Nikod\'{y}m property, then $SE(\mathbb{X})$ is dense in $\mathbb{X}^*$. This density also holds when $\mathbb X$ is ALUR (in particular, when $\mathbb{X}$ is LUR), or more generally when every element of $S_{\mathbb{X}}$ is strongly exposed. On the other hand, if $\mathbb{Y}^*$ has the Radon-Nikod\'{y}m property, then the points having Fr\'{e}chet differentiability of the norm on $\mathbb Y$ are dense in $S_{\mathbb{Y}}$. We refer to \cite{Jung} for further details and examples of Banach spaces for which these density conditions hold.

    \begin{cor}\label{RNP}
          Let $\mathbb{X}, \mathbb{Y}$ be  Banach spaces such that either $\mathbb{X}$ or $\mathbb{Y}^*$ has the RNP. Let $\mathcal{W}$ be a subspace of $\mathbb{L}(\mathbb{X}, \mathbb{Y})$ containing $\mathbb{F}(\mathbb{X}, \mathbb{Y}).$ Then  $\mathcal{W}$ has the BCP if and only if both $\mathbb{X}^*$ and $\mathbb{Y}$ have the BCP.
    \end{cor}
	
	Let us further mention here that as an application of the above theorem, the following results are immediate, which also generalize  \cite[Th. 3.4(ii), Cor. 3.3, Cor. 3.4]{LLLZ}.
	\begin{cor}
		Let $\mathbb{X}, \mathbb{Y}$ be Banach spaces. Then the following results hold true:
		\begin{itemize}
			\item[$(i)$]  Any subspace of $\mathbb{L}(\ell_p, \mathbb{Y})$ containing $\mathcal{F}(\ell_p, \mathbb{Y})$ has the BCP if and only if $\mathbb{Y}$ has the BCP, for any $ p \in [1, \infty]. $ 
			
			\item[$(ii)$] Any subspace of 	$\mathbb{L}(\mathbb{X}, \ell_p)$ containing $\mathbb{F}(\mathbb{X},\ell_p)$ has the BCP if and only if $\mathbb{X}^*$ has the BCP, for any $ p \in  (1, \infty).$ 
			
			\item[$(iii)$] Any subspace of 	$\mathbb{L}(\mathbb{X}, c_0)$ containing $\mathbb{F}(\mathbb{X}, c_0)$ has the BCP if and only if $\mathbb{X}^*$ has the BCP. 
		\end{itemize}
		\end{cor}

The preceding characterization also enables us to settle an open problem raised in \cite[Question 3]{BLS}. More precisely, the authors asked whether $\mathbb{L}(L^p[0,1])$ has the BCP for $1<p<\infty$ and obtained only a partial affirmative answer for $\frac{3}{2}<p<3$ (see \cite[Th. 3.3]{BLS}). Since $L^p[0,1]$ is reflexive for $1<p<\infty$, it has the RNP. Moreover, both $L^p[0,1]$ and its dual space $L^q[0,1]$, where $\frac1p+\frac1q=1$, have the BCP. Therefore, an application of Corollary \ref{RNP} yields the following result, which completely answers Question~3 of \cite{BLS}.

	\begin{cor}
		The space $\mathbb{L}(L_p[0,1])$ has the ball covering property, for any $1 < p < \infty.$
	\end{cor}
	
		\section*{Application to the space of Lipschitz functions}

	Let $(M,d)$ be a  metric space with a distinguished element called $0$ and let $\mathbb{Y}$ be a Banach space. The Lipschitz space $ \operatorname{Lip}_0(M, \mathbb{Y})$ is defined as the space of all Lipschitz map $f: M \to \mathbb{Y}$ such that $f(0)=0,$ endowed with the norm 
	\[
	\|f\|_L= \sup \bigg\{  \frac{\|f(t)- f(s)\|}{d(t,s)}: t,s \in M, t \neq s \bigg\}.
	\]
	Whenever $ \mathbb{Y}=\mathbb{R}$, we denote $\operatorname{Lip}_0(M, \mathbb{R})$ as $\operatorname{Lip}_0(M).$
	For any $t \in M,$ define the evaluation map $\delta_t:  \operatorname{Lip}_0(M) \to \mathbb{R}$ by $\delta_t(f)= f(t).$ The Lipschitz-free space over $M$, denoted by $\mathcal{F}(M)$, is defined as 
	\[
	\mathcal{F}(M)= \overline{span} \{\delta_t: t \in M\} \subseteq \operatorname{Lip}_0(M)^*.
	\]
	It is well-known that dual of $\mathcal{F}(M)$ is isometrically isomorphic to the space $\operatorname{Lip}_0(M).$ Let $\delta_M: M\to\mathcal{F}(M)$ denote the canonical isometric embedding, given by $\delta_M(p)=\delta_p$ for $p\in M$.
      We recall the standard linearization property of Lipschitz-free spaces. There is an onto linear isometry 
      \[
      \operatorname{Lip}_0(M,\mathbb{Y}) \to \mathbb{L}(\mathcal{F}(M),\mathbb{Y}), \quad  \text{ by } 
f \to  T_f,
      \]
where $T_f$ is the unique bounded linear operator satisfying $
T_f\circ\delta_M= f.$
Furthermore, $
\|T_f\|=\|f\|_{\operatorname{L}}, $
and the inverse map is given by $
T \to T\circ\delta_M$ (see \cite{CM}).
	For distinct points $ t, s \in M$, the element $m_{ts}= \frac{\delta_t- \delta_s}{d(t,s)}$ is called the molecule associated with the pair $(t,s).$ The set of all molecules denoted by $\operatorname{Mol}(M). $ We recall the following lemma useful for our study. 
    
    \begin{lemma}\cite[Lemma 4.1]{GPR}
     Let $M$ be a pointed metric space. If $\phi$ is a strongly exposed point of $B_{\mathcal{F}(M)}$, then $\phi \in \operatorname{Mol}(M).$    
    \end{lemma}
  We refer the readers to \cite{GPR} on more about when such a molecule be a strongly exposed points of $B_{\mathcal{F}(M)}.$  
    We say that a Lipschitz function $f \in \operatorname{Lip}_0(M, \mathbb{Y}) $ is of finite-rank if $ dim~span\,f(M)< \infty $ is and the space of all finite-rank Lipschitz map ius denoted by $\operatorname{Lip}_0^F(M, \mathbb{Y}).$

The isometric identification between the spaces $\operatorname{Lip}_0(M, \mathbb{Y})$ and $\mathbb{L}(\mathcal{F}(M), \mathbb{Y})$ allows us to transfer the results obtained above for subspaces of $\mathbb L(\mathbb X,\mathbb Y)$ to the setting of Lipschitz spaces. In particular, since 
\[
f \in\operatorname{Lip}_0^F(M,\mathbb Y)
\quad\Longleftrightarrow\quad
T_f\in\mathbb{F}(\mathcal F(M),\mathbb Y),
\]
the hypotheses appearing in the preceding results have their natural counterparts in terms of $\operatorname{Lip}_0(M)$ and $\operatorname{Lip}_0^F(M,\mathbb Y)$. We, therefore, obtain the following consequences of Theorem \ref{operator:main} for  Lipschitz spaces.
    
	\begin{theorem}
	     Let $M$ be a  pointed metric space and $\mathbb{Y}$ be a Banach space. Let $\mathcal{W}$ be a subspace of $\operatorname{Lip}_0(M, \mathbb{Y})$ containing $\operatorname{Lip}^F_0(M,\mathbb{Y}).$ 	Suppose that either of the following condition holds: 
		\begin{itemize}
			\item[(i)]  $\mathbb{Y}$ has BCP and there exists a countable set of strongly exposed points of $B_{\mathcal{F}(M)}$ separating $\operatorname{Lip}_0(M).$
			
			\item[(ii)] $\operatorname{Lip}_0(M)$ has the BCP and there exists a countable set of weak*-strongly exposed points of $B_{\mathbb{Y}^*}$ separating $\mathbb{Y}^{**}.$
		\end{itemize}
		Then $\mathcal{W}$ has the BCP.
	\end{theorem}

     \begin{cor}\label{cor:lip}
     Let $M$ be a  pointed metric space and $\mathbb{Y}$ be a Banach space. Suppose that either $\mathcal{F}(M)$ or $\mathbb{Y}^*$ has RNP. Let $\mathcal{W}$ be a subspace of $\operatorname{Lip}_0(M, \mathbb{Y})$ containing $\operatorname{Lip}^F_0(M,\mathbb{Y}).$ Then  $\mathcal{W}$ has the BCP if and only if both $\operatorname{Lip}_0(M)$ and $\mathbb{Y}$ have the BCP.
    \end{cor}

   We refer \cite{AGPP} for a detailed study on the RNP of $\mathcal{F}(M).$ Recall from
\cite[Theorem 3.1]{AGPP} that
\[
\mathcal F(M)\text{ has the RNP}
\quad\Longleftrightarrow\quad
\overline{M}\text{ is purely $1$-unrectifiable},
\]
where $\overline{M}$ denotes the completion of $M$. For numerous examples of purely $1$-unrectifiable metric spaces, we refer to \cite[3.3.19-3.3.21]{F} and the references therein.

We now present some conditions for the space  $\operatorname{Lip}_0(M)$ having the BCP, as a direct consequence of Proposition \ref{Proposition: SE(X)} and  
 Remark \ref{remark1}.  

\begin{cor}\label{cor:lip2}
    Let $M$ be a pointed metric space. Then 
    \begin{itemize}
        \item [(i)] If there exists a countable strongly exposed points  of $B_{\mathcal{F}(M)}$ which separates $\operatorname{Lip}_0(M).$ Then $\operatorname{Lip}_0(M)$ has the BCP.

        \item[(ii)] Whenever $\mathcal{F}(M)$ has the RNP. Then $\operatorname{Lip}_0(M)$ has the BCP if and only if there exists a countable set of strongly exposed points of $B_{\mathcal{F}(M)}$ separating $\operatorname{Lip}_0(M).$
    \end{itemize}
\end{cor}

   We next record some concrete situations in which the hypotheses above are
satisfied. We first recall that, for $x,y,z\in M$, the Gromov product of $x$
and $y$ at $z$ is defined by
\[
(x,y)_z
=
\frac{1}{2}\big(d(x,z)+d(y,z)-d(x,y)\big).
\]
 A pointed metric space $M$ is said to be \emph{Gromov concave} \cite{CM} if, for every pair of distinct points $x,y\in M$, there exists $\varepsilon_{x,y}>0$ such that
\[
(x,y)_z
>
\varepsilon_{x,y}\min\{d(x,z),d(y,z)\}
\]
for every $z\in M\setminus\{x,y\}$. It is known that $M$ is Gromov concave if and only if every molecule of $M$ is a strongly exposed point of $B_{\mathcal F(M)}$ \cite[Theorem 5.4]{GPR}. 
Thus, for a Gromov concave metric space, we have the following result.

\begin{prop}
    Let $M$ be a separable pointed Gromov concave metric space. Then $\operatorname{Lip}_0(M)$ has the BCP. 
\end{prop}
\begin{proof}
Let $\{t_n:n\in\mathbb N\}$ be a countable dense subset of $M$. Then $
\mathcal{A}=\{m_{t_nt_k}:n,k\in\mathbb N,\ n\neq k\} $
is a countable family of molecules separating $\operatorname{Lip}_0(M)$.
Since $M$ is Gromov concave, every molecule is a strongly exposed point
of $B_{\mathcal F(M)}$. Hence $\mathcal A$ is a countable family of
strongly exposed points of $B_{\mathcal F(M)}$ separating
$\operatorname{Lip}_0(M)$. Therefore, by Remark~\ref{remark1},
$\operatorname{Lip}_0(M)$ has the BCP.
\end{proof}

Following \cite[Th. 2.1]{CGMR}, the unit sphere of the Euclidean plane endowed with the inherited Euclidean
metric is a Gromov concave space.

\section*{Stability of BCP under the $p$-sum.}

We next study the stability of the BCP under the $p$-sum of Banach spaces. 
Let $\{\mathbb{X}_n\}_{n \in \mathbb{N}}$ be a sequence of Banach spaces. 
Then for any $1 \leq p < \infty,$
\[
\Big(\sum_{n=1}^{\infty}\oplus \mathbb{X}_n\Big)_p
:= \bigg\{ \widetilde{x}= (x_1, x_2, \ldots): x_i \in \mathbb{X}_i, \,
\sum_{i=1}^{\infty} \|x_i\|^p < \infty\bigg\},
\]
\[
\Big(\sum_{n=1}^{\infty}\oplus \mathbb{X}_n\Big)_\infty
:= \bigg\{ \widetilde{x}= (x_1, x_2, \ldots): x_i \in \mathbb{X}_i, \,
\sup \{\|x_i\|: i \in \mathbb{N}\} < \infty\bigg\},
\]
\[
\Big(\sum_{n=1}^{\infty}\oplus \mathbb{X}_n\Big)_{c_{00}}
:= \bigg\{ \widetilde{x}= (x_1, x_2, \ldots): x_i \in \mathbb{X}_i, \, x_i=0 \, \text{for all but finitely many i}\bigg\}.
\]
For any $\widetilde{x} \in \Big(\sum_{n=1}^{\infty}\oplus \mathbb{X}_n\Big)_p,$ $
\|\widetilde{x}\|= \left(\sum_{i=1}^{\infty} \|x_i\|^p\right)^{1/p},$
and whenever $\widetilde{x} \in \Big(\sum_{n=1}^{\infty}\oplus \mathbb{X}_n\Big)_\infty,$ $
\|\widetilde{x}\|= \sup_{i\in \mathbb{N}} \|x_i\|.$

In \cite{LZ}, the authors have established the stability of the BCP for product spaces equipped with the $p$-norm. It was shown that 
$\Big(\sum_{n=1}^{\infty}\oplus \mathbb{X}_n\Big)_p$ has the BCP if and only if each $\mathbb{X}_n$ has the BCP. Moreover, it was observed that the ball covering centers can be chosen from 
$\Big(\sum_{n=1}^{\infty}\oplus \mathbb{X}_n\Big)_{c_{00}}$.

Here we provide an alternative proof of the same result, \emph{along with a refinement of the choice of the centers,} by using the characterization of the BCP obtained in Lemma~\ref{lemma:main}. Indeed, we show that whenever $1 < p \leq \infty$, the ball covering centers can be chosen among elements having exactly one nonzero coordinate. To be more precise, for $x \in \mathbb{X}_k,$ we denote by
\[
x^{(k)}:=(0,\ldots,0,x,0,\ldots)
\]
the element of $\Big(\sum_{n=1}^{\infty}\oplus \mathbb{X}_n\Big)_p$ whose $k$-th coordinate is $x$ and the remaining coordinates are zero. Our proof shows that the ball covering centers may be taken from the collection of such elements. On the other hand, our method also shows that when $p=1$, it is not possible to choose all ball covering centers from this collection. 

\begin{theorem}\label{theorem:Countable Banach space}
Let $ 1< p\leq \infty $ and let $\{\mathbb{X}_n\}_{n\in\mathbb{N}}$ be a sequence of Banach spaces.
Let $
\mathcal{X}=\Big(\sum_{n=1}^{\infty}\oplus \mathbb{X}_n\Big)_p .$
Then $\mathcal{X}$ has the BCP if and only if each
$\mathbb{X}_n$ has the BCP. Moreover, the ball covering centers for $ \mathcal{X} $ can be chosen among elements of $ \mathcal{X} $ having exactly one nonzero coordinate.
\end{theorem}

\begin{proof}
Let us first prove the necessary part.
Since $\mathcal{X}$ has the BCP, it follows from Lemma $\ref{lemma:main}$ that there exists a
sequence $\{\widetilde{x}_i\}_{i\in\mathbb{N}}\subset \mathcal{X}\setminus\{0\}$ such that
\[
\bigcap_{i\in\mathbb{N}}\widetilde{x}_i^{\perp}=\{0\}, \quad \text{where each} \, \,  \widetilde{x}_i=(x_1^i,x_2^i,\ldots) \, \text{with} \, \, x_j^i\in\mathbb{X}_j.
\]
Let $k \in \mathbb{N} $ be arbitrary but fixed. We  will show that $ \mathbb{X}_{k} $ has the BCP. Consider the sequence $\{ x_{k}^i\}_{i \in \mathbb{N}} \subset \mathbb{X}_{k}. $ For each $ i \in \mathbb{N} $ if $ x_{k}^i= 0, $  then it will contradict that $ \mathcal{X} $ has the BCP. Let  $y \in \cap_{i \in \mathbb{N}} (x_{k}^i)^\perp.$ Now consider the element $ \widetilde{y}= (0, 0, \ldots, 0, y, 0, \ldots) \in \mathcal{X},$ where $y$ appears in the $k$-th coordinate. Observe that for any $ i \in \mathbb{N}$ and $ 1 < p < \infty,$
\begin{eqnarray*}
\| \widetilde{x_i} + \lambda \widetilde{y}\|^p &=& \| ( x_1^i, x_2^i, \ldots, x_{k-1}^i, x_{k}^i + \lambda y , x_{k+1}^i, \ldots) \|^p\\
&=& \sum_{j \in \mathbb{N}, j \neq k} \| x_j^i\|^p + \|x_{k}^i + \lambda y\|^p\\
&\geq & \sum_{j \in \mathbb{N}, j \neq k} \| x_j^i\|^p + \|x_{k}\|^p \\&=& \|\widetilde{x_i}\|^p.
\end{eqnarray*}
So, for any scalar $\lambda$, $\| \widetilde{x_i} + \lambda \widetilde{y}\| \geq \| \widetilde{x_i}\|. $ In other words, $\widetilde{y} \in \widetilde{x_i}^\perp,$ for any $i 
\in \mathbb{N}.$ This implies $\widetilde{y}=0 $ and so, $y=0.$ Therefore, $\cap_{i \in \mathbb{N}} (x_{k}^i)^\perp = \{0\}.$ Following Lemma \ref{lemma:main}, $\mathbb{X}_k$ has the BCP. Since $ k \in \mathbb{N} $ is arbitrary, it establishes the necessary part for $ 1 \leq p < \infty. $ For $p = \infty$, we can provide similar argument with some notational modifications which are self-evident. This establishes the necessary part of the theorem.\\

Conversely, suppose that each $\mathbb{X}_k$ has the BCP. 
Then for every $k\in\mathbb{N},$ there exists a countable set 
$\{x_k^i\}_{i\in\mathbb{N}} \subset \mathbb{X}_k \setminus \{0\}$ such that
\[
\bigcap_{i\in\mathbb{N}} (x_k^i)^{\perp} = \{0\}.
\]
For each $k \in \mathbb{N}$, define $
S_k
=
\big\{ (0,0,\ldots,0,x_k^i,0,\ldots) \in \mathcal{X} : i \in \mathbb{N} \big\},$
where $x_k^i$ appears in the $k$-th coordinate.
Now define
\[
\mathcal{S} = \bigcup_{k \in \mathbb{N}} S_k.
\]
Clearly, $\mathcal{S}$ is a countable subset of $\mathcal{X}$.
Let $\tilde{w}=(w_1,w_2,\ldots)\in\mathcal{X}$ be such that $
\tilde{w} \in \bigcap_{s\in\mathcal{S}} s^{\perp}.$
We claim that $\tilde{w}=0$. Let $k \in \mathbb{N}$ be arbitrary but fixed.
For any $i \in \mathbb{N},$ let $\widetilde{z_i}= (0, 0, \ldots, x_{k}^i, 0, \ldots) \in S_{k}.$  Let us now consider the following two cases: 

\textbf{Case 1:} Let $1 < p < \infty. $  Observe that for any  scalar $\lambda,$
\[
\|\widetilde{z_i}+ \lambda \widetilde{w}\|^p= \| ( \lambda w_1, \lambda w_2, \ldots, \lambda w_{k-1}, x_{k}^i + \lambda w_{k}, w_{k+1}, \ldots)\|^p = \sum_{j \neq k} \|\lambda w_i\|^p + \|x_{k}^i+ \lambda w_{k}\|^p.
\]
As $\widetilde{w} \in \widetilde{  z_i}^{\perp},$ it follows that for any scalar $\lambda$, 
\[
\|\widetilde{z_i}+ \lambda \widetilde{w}\| \geq \|\widetilde{z_i}\|= \|x_k^i\|.
\]
Therefore, for any scalar $\lambda$, we obtain
\begin{equation}\label{eqn:1}
\sum_{j \neq k} \|\lambda w_j\|^p + \|x_{k}^i+ \lambda w_{k}\|^p \geq \|x_k^i\|^p.
\end{equation}
We claim that for any scalar $\lambda,$ $\|x_k^i + \lambda w_k \| \geq \|x_k^i\|.$ Suppose on the contrary that there exists $\lambda_0$ such that $\|x_k^i+ \lambda_0 w_k\|< \|x_k^i\|.$ Let $\epsilon= \|x_k^i\|^p- \|x_k^i+ \lambda_0 w_k\|^p> 0.$ Using a standard convexity argument, we get that for any natural number $n > 1,$ 
\[
\|x_k^i + \frac{\lambda_0}{n} w_k \|^p \leq \|x_k^i\|^p - \frac{\epsilon}{n}.
\]
 Now take $n_0 \in \mathbb{N}$ such that $n_0^{p-1} > \frac{1}{\epsilon} |\lambda_0|^{p-1} \sum_{j \neq k} \|w_j\|^p. $
 Therefore, 
 \begin{eqnarray*}
     \sum_{j \neq k} \|\frac{\lambda_0}{n_0} w_j\|^p + \|x_{k}^i+ \frac{\lambda_0}{n_0} w_{k}\|^p &\leq& \frac{|\lambda_0|^p}{n_0^p} \sum_{j \neq k} \| w_j\|^p + \|x_{k}^i\| - \frac{\epsilon}{n_0}\\ & < & \|x_k^i\|,
 \end{eqnarray*}
 which contradicts the inequality (\ref{eqn:1}). Therefore, for any scalar $\lambda,$ $\|x_k^i + \lambda w_k\| \geq \|x_k^i\|.$

 \textbf{Case 2:} Let $p=\infty.$  Since $\widetilde{w} \in \widetilde{  z_i}^{\perp},$ it follows that for any scalar $\lambda$,
 \begin{equation}\label{eqn:2}
     \|\widetilde{z_i} + \lambda \widetilde{w} \|_{\infty}= \sup \{ \| \lambda w_j\|, \|x_k^i+ \lambda w_k\|: j \neq k, j \in \mathbb{N}\} \geq \|x_k^i\|.
 \end{equation}
Proceeding as before, suppose on the contrary that there exists a scalar $\lambda_0$ such that $\|x_k^i+ \lambda_0 w_k\|< \|x_k^i\|$.  So, for any natural number $n > 1,$ $\|x_k^i + \frac{\lambda_0}{n} w_k\| \leq \|x_k^i\|.$  Now take $n_0 \in \mathbb{N}$ such that $ n_0 > \frac{1}{\|x_k^i\|} |\lambda_0| \sup \{\|w_j\|: j\in \mathbb{N}, j \neq k\}.$ Therefore, we obtain
\[
\sup \bigg\{ \|\frac{\lambda_0}{n_0} w_j\|, \|x_k^i+ \frac{\lambda_0}{n_0} w_k\|: j \neq k, j \in \mathbb{N} \bigg\} < \|x_k^i\|.
\]
This contradicts the inequality (\ref{eqn:2}) and hence, for any scalar $\lambda,$ $\|x_k^i + \lambda w_k\| \geq \|x_k^i\|.$
So, for any $1 < p \leq \infty,$ we have  $w_{k} \in (x_{k}^i)^\perp.$  
 Since $i \in \mathbb{N}$ is arbitrary, we obtain $$ w_{k} \in \cap_{i \in \mathbb{N}} (x_{k}^i)^\perp \implies w_{k}=0. $$  
Hence $w_k=0$ for every $k\in\mathbb{N}$, and therefore $\tilde{w}=0$.
Consequently,
\[
\bigcap_{s\in\mathcal{S}} s^{\perp}=\{0\},
\]
and following Lemma \ref{lemma:main},
$\mathcal{X}$ has the BCP. Moreover, from the arguments given in the sufficient part of the proof, it is evident that the ball covering centers can be taken from the set $\mathcal{S}.$
This completes the proof.
	\end{proof}

In light of the above result, it is important to note the specialty of the case $ p = 1. $ Note that for $p=1,$ all the ball covering centers can not be chosen from the set $\mathcal{S}.$ This is clear from the fact that even though the space $\ell_1$ has the BCP, we still have $ x= ( -\frac{1}{2}, \frac{1}{4}, \frac{1}{8}, \frac{1}{16}, \ldots ) \in \ell_1 $ such that $ x \in \cap_{i \in \mathbb{N}} e_i^\perp. $ Therefore, we obtain from Lemma \ref{lemma:main} that $\ell_1 $ does not admit the BCP with centers chosen from the $e_i$'s and their scalar multiples only. The following remark further elucidates this observation.

\begin{remark} \label{theorem:1-sum}
Let $\{\mathbb{X}_n\}_{n\in\mathbb{N}}$ be a sequence of Banach spaces.
Let $ \mathcal{X}=\Big(\sum_{n=1}^{\infty}\oplus \mathbb{X}_n\Big)_1 .$
As before, it is easy to see that if $\mathcal{X}$ has the BCP, then each $\mathbb{X}_n$ has the BCP. On the other hand, suppose that each
$\mathbb{X}_n$ has the BCP. Then For each $k\in\mathbb{N}, $ there exists a countable set 
$\{x_k^i\}_{i\in\mathbb{N}} \subset \mathbb{X}_k \setminus \{0\}$ such that
\[
S_{\mathbb{X}_k} \subset \bigcup_{i,j\in \mathbb{N}} B(n_j x_k^i, n_j \|x_k^i \|). 
\] 
 Take any $ \tilde{w}= (w_1,w_2,\dots,w_n,\ldots) \in S_{\mathcal{X}}, $ where $ w_i \in \mathbb{X}_i \setminus \{0\} $ for each $ i \in \mathbb{N}. $ 
Since $\mathbb{X}_t$ has the BCP for any $ t \in \mathbb{N},$  it is clear that there exist $x_t^{i_t} \in \mathbb{X}_t$ and $n_{t} \in \mathbb{N}$ such that 
$\|n_t x_t^{i_t} - w_t\| < \|n_t x_t^{i_t}\|.$ Let $\epsilon_t > 0$ such that 
\[
\|n_t x_t^{i_t} - w_t\| = \|n_t x_t^{i_t}\|- \epsilon_t.
\]
As $\lim_{n \to \infty} \|w_n\| =0,$ there exists $N \in \mathbb{N}$ such that 
\[
\sum_{k > N} \|w_k\| < \sum_{t=1}^{N} \epsilon_t.
\]
Now, taking $z= ( n_1x_1^{i_1}, n_2 x_2^{i_2} , \ldots, n_{N} x_{N}^{i_N}, 0, 0, \ldots) \in \mathcal{X}, $ we get that,
\begin{eqnarray*}
    \|z- \widetilde{w}\|_1 &=& \sum_{t=1}^{N} \| n_t x_t^{i_t} - w_t\| + \sum_{ k > N} \|w_k\|\\
    &=& \sum_{t=1}^N \|n_t x_t^{i_t} \| - \sum_{t=1}^N \epsilon_t + \sum_{k > N} \|w_k\|\\
    &< & \|z\|.
\end{eqnarray*}
Consider $\mathcal{S}$ to be the set of all finitely supported sequences of the form
\[
(n_1 x_1, \dots, n_j x_j, 0, 0, \dots) \in \mathcal{X},
\]
where $j \in \mathbb{N}$, $x_k \in \{x_k^i\}_{i \in \mathbb{N}}$, and $n_k \in \mathbb{N}$ for each $1 \le k \le j$.
This implies that $z \in \mathcal{S}$ and $ \widetilde{w} \in B(z, \|z\|).$ Since $\widetilde{w}$ is arbitrary, we get that  $ S_{\mathcal{X}} \subset \bigcup_{z\in \mathcal{S} } B(z,\|z\|). $ Therefore, each ball covering center can be chosen among elements having finitely many nonzero coordinates.
 \end{remark}


\section*{Section-II}

In this section, our primary goal is to study the minimal ball covering number of 
$\mathbb{L}(\mathbb{X}, \mathbb{Y})$ in the finite-dimensional setting. 
To this end, we first establish a characterization of finite ball coverings 
in a finite-dimensional Banach space in terms of Birkhoff-James orthogonality 
cones. This result will serve as a fundamental tool for all subsequent results in this section.

\begin{lemma}\label{lemma:orthogonality}

Let $\mathbb{X}$ be a finite-dimensional Banach space and let 
$\mathcal{A}\subset \mathbb{X}\setminus\{0\}$ be a compact set with the property that 
for every nonzero $x\in\mathbb{X}$ there exists $r_x>0$ such that $r_x x\in\mathcal{A}$. 
Then the following statements are equivalent:
\begin{itemize}
  
		\item[$(i)$] $ \mathcal{A} $ admits a ball covering consisting of $ m $ balls,
		
		\item[$(ii)$] There exist $ x_1, x_2, \ldots, x_m \in Sm(\mathbb{X}) $ such that $ \cap_{i=1}^m x_i^- = \{0\}, $

		\item[$(iii)$] There exist $ y_1, y_2,\ldots, y_m \in \mathbb{X} \setminus \{0\} $ such that $ \cap_{i=1}^m y_i^- = \{0\}, $

		
		\item[$(iv)$] There exist $ z_1, z_2,\ldots, z_m \in \mathbb{X} \setminus \{0\} $ such that $ \cap_{i=1}^m z_i^+ = \{0\}, $

		\item[$(v)$] There exist $ f_1, f_2,\ldots, f_m \in S_{\mathbb{X}^*} $ exposed points of $ B_{\mathbb{X}^*} $ such that $$ \cap_{i=1}^m \{ y \in \mathbb{X} : f_i(y) \leq 0 \} = \{0\}. $$ 
		
	\end{itemize}
	
\end{lemma}

\begin{proof}
	$(i) \implies (ii):$ Since $ \mathbb{X} $ is a finite-dimensional Banach space and $ \mathcal{A} $ admits a ball covering consisting of $ m $ balls, it follows from \cite [Prop. 2.7]{CCS} that there exist $ x_1,x_2,\ldots,x_n \in Sm(\mathbb{X}) $ such that $ \mathcal{A} \subset \cup_{i=1}^m B(x_i, r_{x_i}), $ where $ 0< r_{x_i} \leq \|x_i\| $ for each $ 1\leq i \leq m. $  So, for any $ y \in \mathcal{A} $ there exists $ x_j \in \{ x_1,x_2,\ldots,x_m \} $ such that $ y \in B(x_j, r_{x_j}). $ This implies that $ \| x_j -y \|< \|x_j \|, $ equivalently $ y \not\in x_j^-. $ Therefore, $ y \not\in \cap_{i=1}^m x_i^-. $ Since $ y\in \mathcal{A} $ is arbitrary, it follows that $ \mathcal{A} \cap ( \bigcap_{i=1}^m x_i^- ) = \emptyset. $ Then using the homogeneity of Birkhoff-James orthogonality, it is easy to verify that $ \cap_{i=1}^m x_i^- = \{0\}. $

	$(ii) \implies (iii)$ is immediate.

		 $(iii) \implies (i):$ As $ \cap_{i=1}^m y_i^- = \{0\}, $  for any $ x \in \mathcal{A} , $ there exists $ j \in \{ 1, 2,\ldots, m \} $ such that $ x \not\in y_j^-. $ This implies that there exists $ \lambda_x >0 $ such that $ \| y_j - \lambda_x x \| <\|y_j\|. $ Choosing $ \lambda_x = \frac{1}{t_x}, $ we get that $ \| t_x y_j - x \| < t_x \| y_j \|. $ Consequently, $ x \in B(t_x y_j, r_x ), $ for some $ 0< r_x \leq t_x \| y_j \|. $ Clearly, $ \mathcal{U} := \{ B( t_{x} y_j, r_x ): x \in \mathcal{A} \} $ is an open cover of $ \mathcal{A}. $ Since $ \mathcal{A} $ is compact, it has a finite sub-cover $ \mathcal{\widetilde{U}} = \{ B( t_{x_p} y_{x_p}, r_{x_p}): 1 \leq p \leq m_0 \}, $ where $ m_0 \in \mathbb{N}  $ and $ x_p \in \mathcal{A},$ for each $ 1 \leq p \leq m_0. $  We note that $ y_{x_p} \in \{ y_1,y_2,\ldots,y_m \} ,$  for each $ 1 \leq p \leq m_0. $ Now applying Proposition \ref{prop:compact}, for each $1 \leq i \leq m,$ there exists $\lambda_i > 0$ and $0 < r_i \leq \lambda_i \|y_i\|$ such that $ \mathcal{A} \subset \cup_{i=1}^m B(\lambda_i y_i, r_i).$

		$(iii) \iff (iv)$ follows directly   by taking $ z_i = -y_i $ for each $ 1\leq i \leq m. $

	 $(ii)  \iff  (v)$ follows from the observation that given any $ x \in Sm(\mathbb{X}), $ there exists  $ f \in Exp(B_{\mathbb{X}^*}) $  and $ x^- = \{ y \in \mathbb{X} : f(y) \leq 0 \}. $ This completes the proof.

\end{proof}

\begin{remark}
The preceding lemma admits the following immediate refinements. Under its assumptions, the set $A$ admits a minimal covering by $m$ balls if and only if there exist $x_1, \dots, x_m \in \mathbb{X} \setminus \{0\}$ such that $\bigcap_{i=1}^m x_i^{-} = \{0\}$, while for every $1 \le k < m$ and every choice of $y_1, \dots, y_k \in \mathbb{X} \setminus \{0\}$, one has $\bigcap_{i=1}^k y_i^{-} \neq \{0\}$. An equivalent formulation holds true with vectors from the unit sphere. Moreover, if $A \subset \mathbb{X} \setminus \{0\}$ is compact (without the absorbing assumption), then $A$ admits a covering by $m$ balls if and only if there exist $x_1, \dots, x_m \in S_X$ such that $A \cap \bigcap_{i=1}^m x_i^{-} = \emptyset$, with analogous formulations in terms of smooth points, positive cones, and exposed functionals as in the lemma.
\end{remark}

As an interesting applications of Lemma \ref{lemma:orthogonality}, we next obtain an upper bound for the ball covering number of the $ p $-sum of Banach spaces, for $1 \leq p < \infty$, in terms of the ball covering numbers of $\mathbb{X}$ and $\mathbb{Y}$.
To this end, we first recall the following well-known characterization of exposed points of the dual unit ball of $p$-sums of Banach spaces.

\begin{lemma}\label{p-sum}
	
	Let $\mathbb{X}$ and $\mathbb{Y}$ be Banach spaces. Then $$ Exp(B_{\mathbb{X}^* \oplus_p \mathbb{Y}^*}) = \{ (f, g): \|f\|^p + \|g\|^p = 1,  f \in Exp( B_{X^*}), g \in Exp( B_{Y^*})\}. $$
	
\end{lemma}

\begin{theorem}\label{p sum}
	
	Let $ \mathbb{X} $ and $ \mathbb{Y} $ be finite-dimensional Banach spaces such that $ \mathbb{X}, \mathbb{Y} $ admits a ball covering consisting of $ m $ and $ n $ balls, respectively. Let $ 1\leq p <\infty.$ Then $S_{\mathbb{X} \oplus_p \mathbb{Y}} $ admits a ball covering consisting of $ m+n+1 $ balls. 
	
\end{theorem}

\begin{proof}
	
  Since   $ \mathbb{X}, \mathbb{Y} $ admits a ball covering consisting of $ m $ and $ n $ balls, respectively, it follows from  Lemma \ref*{lemma:orthogonality} that there exist $ f_1,f_2,\ldots,f_m  \in  Exp( B_{X^*}) $ and  $ g_1,g_2\ldots,g_n \in Exp( B_{Y^*}) $ such that 
  \begin{eqnarray*}
      \bigcap_{i=1}^m \{ x \in \mathbb{X} : f_i(x) \leq 0 \} = \{0\} \quad \text{and} \quad 
      \bigcap_{i=1}^n \{ y \in \mathbb{Y} : g_i(y) \leq 0 \} = \{0\}.
  \end{eqnarray*}

By Lemma~\ref{p-sum}, the set of functionals $
\{(f_1,g_1), (f_i,-g_1),(-f_1,g_j) : 1 \leq i \leq m, 1 \leq j \leq n\} \subset  
Exp(B_{\mathbb{X}^*\oplus_q \mathbb{Y}^*}). $ We claim that
\[
\bigcap_{i=1}^m (f_i,-g_1)^-
\ \cap\
(f_1,g_1)^-
\ \cap\
\bigcap_{j=1}^n (-f_1,g_j)^-
=
\{0\}.
\]
Let $(x,y)\in \mathbb{X}\oplus_p \mathbb{Y}$ belong to the above intersection. Then
\[
f_1(x)+g_1(y)\le 0
\quad\text{and}\quad
f_1(x)-g_1(y)\le 0,
\]
which together imply $f_1(x)\le 0$. Moreover, for each $1\le j\le n$,
\[
-f_1(x)+g_j(y)\le 0,
\]
and hence each $g_j(y)\le 0$. Consequently,
\[
y\in \bigcap_{j=1}^n \{y\in\mathbb{Y}: g_j(y)\leq 0\} \implies y =0.
\]
Also, since  $(x,y)\in (f_i,-g_1)^-$ for each 
$1\leq i\leq m$, we obtain each  $f_i(x)\leq 0$ and hence
\[
x\in \bigcap_{i=1}^m \{x\in\mathbb{X}: f_i(x)\le 0\} \implies x=0.
\]
This completes the proof of the theorem.
\end{proof}

In \cite{C}, Cheng showed that whenever an $ n $-dimensional Banach space $ \mathbb{X} $ is smooth, it follows that $\mathcal{B}_{\min}^{\#} = n+1. $ Moreover, the smoothness assumption is not necessary for an $ n $-dimensional Banach space to have such a covering by $n+1$ balls. We next present a sufficient condition for an $ n $-dimensional Banach space to have the ball covering number equal to $ n+1. $ As an application of this, we will show that the space $\ell_{1}^n$ and some operator spaces have the minimal   ball covering number equal to $ n+1. $ It follows from \cite{C} that this number is the minimal possible  for an $ n $-dimensional Banach space. For any $ x_1,x_2,\ldots,x_k \in \mathbb{X}, $ we define $$ cone(x_1, x_2,\ldots, x_k):= \bigg\{ \sum_{i=1}^{k} \alpha_i x_i: \alpha > 0,  \quad   1 \leq i \leq k \bigg\}. $$

\begin{theorem}\label{theorem:sufficient condition}	
	Let $ \mathbb{X} $ be an $ n $-dimensional Banach space and let $ \mathcal{A}\subset \mathbb{X} \setminus \{0\} $ be a compact set. Suppose that there exist $ f_1,f_2,\ldots,f_n, f_{n+1} \in Exp( B_{\mathbb{X}^*}) $ such that   the following holds:
	 
	\begin{itemize}
		\item[$$(i)$$] $span \{f_1, f_2, \ldots, f_n\}= \mathbb{X}^*,  $
		
		\item[$(ii)$] $ f_{n+1} \in cone(-f_1,-f_2,\ldots,-f_n). $ 
			\end{itemize}
		Then $ \mathcal{A} $ admits a ball covering consisting of $ n+1 $ balls.
	
\end{theorem}

\begin{proof}
	
 Following \cite[Th, 2.5]{Sain25}, it is sufficient to show that for any  $ x\in \mathcal{A} , $  
	$$ \max\{f_i(x): 1\leq i \leq n+1 \}>0.$$
Let $x \in \mathcal{A} $.  If $ f_i(x) >0 $ for some $ 1\leq i\leq n, $ then we are done. Suppose that $f_i(x) \leq 0, $ for any $1 \leq i \leq n.$ As $f_{n+1} \in cone( -f_1,-f_2,\ldots,-f_n),$ we have $f_{n+1}(x) \geq 0.$ Observe that $f_{n+1}(x) =0 \implies f_i(x)= 0 $ for every $1 \leq i \leq n.$ 	Since $span \{f_1, f_2, \ldots, f_n\}= \mathbb{X}^*  $, it follows that $ \cap_{i=1}^n ker f_i = \{0\}. $ Therefore, $f_{n+1}(x) > 0, $  completing the proof.

\end{proof}

  As an immediate consequence of the above result, we have the following observation on the minimal ball coverings of smooth Banach spaces. 

\begin{cor}\label{smooth space}
	 Let $ \mathbb{X} $ be an $ n $-dimensional smooth Banach space and let $ \mathcal{A} \subset \mathbb{X} \setminus \{0\} $ be a compact set. Then $ \mathcal{A} $ admits a ball covering consisting of $ n+1 $ balls.\\
	In particular, $  S_\mathbb{X} $ admits a ball covering consisting of $ n+1 $ balls.
	
\end{cor}

\begin{proof}
Let  $f_1, f_2, \ldots, f_n \in B_{\mathbb{X}^*}$ be linearly independent and let $ f_{n+1} = - \frac{1}{\| \sum_{i=1}^{n} f_i \| } \sum_{i=1}^{n} f_i. $	As $ \mathbb{X} $ is finite-dimensional and smooth, it is clear that $ \mathbb{X}^* $ is strictly convex. So, $f_i \in Exp(B_{\mathbb{X}^*}),$ for any $1 \leq i \leq n+1.$ Therefore, using Theorem \ref{theorem:sufficient condition}, $ \mathcal{A} $ admits a ball covering consisting of $ n+1 $ balls.\\
The remaining part follows trivially from the above observation. 
\end{proof}

  As a consequence of Theorem \ref{theorem:sufficient condition}, we obtain  the minimal ball covering number of any compact set $ \mathcal{A} $ in $\ell_1^n,$ that does not contain the origin.  Clearly, this is a generalization of the earlier existing result \cite[Page 1669]{CLZZ} for $ S_{\ell_1^n}.$ 

\begin{cor}\label{1 norm}
	 
	  Let $ \mathcal{A} \subset \ell_1^n \setminus \{0\} $ be  a compact set, where $n \geq 3.$ Then $ \mathcal{A} $ admits a ball covering consisting of $ n+1 $ balls.  
	
	  In particular, $ S_{\ell_1^n} $ admits a ball covering consisting of $ n+1 $ balls, for any $ n\geq 3$. 
	 
\end{cor}

\begin{proof}
	 Let $ \psi: (\ell_1^n)^* \longrightarrow \ell_{\infty}^n $ be the canonical isometric isomorphism. Consider  $ f_1,f_2,\ldots,f_n \in \mathbb{X}^*$ such that  
	 $$\psi(f_i)= (-1,-1,\dots, \underset{i-th~ position}{1}, \dots, -1 ), \quad \forall 1 \leq i \leq n.$$
	  Let $f_{n+1} \in \mathbb{X}^*$ such that $\psi( f_{n+1})= (1,1,\ldots,1).$ Clearly, $$f_{n+1} = \frac{1}{n-2} \sum_{i=1}^{n} (-f_i)  . $$ So, $ f_{n+1}  \in B_{(\ell_{1}^n)^*} $ and $ f_{n+1} \in cone(-f_1,-f_2,\ldots,-f_n). $ As each $f_i\in Exp(B_{(\ell_{1}^n)^*}),$ following Theorem \ref{theorem:sufficient condition}, we obtain the desired result.
	    \end{proof}

Next, we study the ball covering number in the space of all bounded linear operators on finite-dimensional Banach spaces. For that purpose, we need the following lemma which is analogous to Lemma \ref{lemma:orthogonality} for operator spaces.

\begin{lemma}\label{lemma:operator1}
	Let $\mathbb{X}, \mathbb{Y}$ be finite-dimensional Banach spaces. Then the following are equivalent:
	\begin{itemize}
		\item[$(i)$] $S_{\mathbb{L}(\mathbb{X}, \mathbb{Y})}$ admits a ball covering consisting of $ m $ balls,
		
		\item[$(ii)$] There exists $x_i \in Exp(B_{\mathbb{X}}), \,  y_i^* \in Exp(B_{\mathbb{Y}^*})  $ such that the system of $m$ number of inequalities $$ y_i^*(Ax_i) \leq    0, \quad 1 \leq i \leq  m$$ has the unique solution $ A=0.$
		
	\end{itemize}
\end{lemma}

\begin{proof}
	$(i) \implies (ii):$ As  $S_{\mathbb{L}(\mathbb{X}, \mathbb{Y})}$ has a covering by $m$ balls, following Lemma \ref{lemma:orthogonality} there exists  $T_i \in Sm(\mathbb{L}(\mathbb{X}, \mathbb{Y})), 1 \leq i \leq m$ such that $\cap_{i =1}^m T_i^- =0.$ As each $T_i$ is smooth, following \cite[Th. 3.1]{SPMR} we get that  $M_{T_i}=\{ \pm x_i\}, $ for some $x_i \in S_{X}$ and each $Tx_i\in S_Y$ is smooth. Let $J(Tx_i)=\{ y_i^*\},$ for any $ 1\leq i \leq m.$ Clearly, each $y_i^* \in Exp(B_{Y^*}).$ It is easy to observe that $J(T_i)= \{ \pm [y_i^* \otimes x_i]\},$ where $[y_i^* \otimes x_i](A)= y_i^*(Ax_i)$  for all $A \in \mathbb{L}(\mathbb{X}, \mathbb{Y}).$
Therefore, following Lemma  \ref{functional}, 
	\[
	T_i^-= \{ A \in \mathbb{L}(\mathbb{X}, \mathbb{Y}): [y_i^* \otimes x_i](A) \leq 0\} =  \{ A \in \mathbb{L}(\mathbb{X}, \mathbb{Y}): y_i^*(Ax_i) \leq 0\}.
	\] 
	Since $\cap_{i =1}^m T_i^- = \{0\},$ we obtain the required result.
	
	$(ii)  \implies (i):$ As $x_i \in Exp(B_{\mathbb{X}}), y_i^* \in Exp(B_{\mathbb{Y}^*}),$ there exist $x_i^* \in S_{\mathbb{X}^*}$ and $y_i \in S_{\mathbb{Y}}$ such that $M_{x_i^*}= \{ \pm x_i\}$ and $J(y_i)=\{y_i^*\}.$ For any $1 \leq i \leq m,$ define 
    $$T_i(x) = x_i^*(x) y_i,\quad \text{for any} \, x \in \mathbb{X}.$$
    It is easy to show that each $T_i$ is smooth and $J(T_i)= \{ y_i^* \otimes x_i\}.$ Again, applying Lemma \ref{functional}, we have that $\cap_{i=1}^{m}T_i^-= \{ A \in \mathbb{L}(\mathbb{X}, \mathbb{Y}): y_i^*(Ax_i) \leq 0\}=\{0\}.$ Following Lemma \ref{lemma:orthogonality}, $S_{\mathbb{L}(\mathbb{X}, \mathbb{Y})}$ admits a ball covering consisting of $ m $-balls.
\end{proof}

\begin{prop}
	Let $\mathbb{X}$ be an $m$-dimensional Banach space and let $\mathbb{Y}$ be an $n$-dimensional Banach space. Suppose that $S_{\mathbb{X}^*}$ and $S_{\mathbb{Y}}$ admit a ball covering consisting of $p$ and $q$ number of balls, respectively. Then $S_{\mathbb{L}(\mathbb{X}, \mathbb{Y})}$ admits a ball covering consisting of $k$ balls, where $k = \min \{ mq, np\}.$
\end{prop}

\begin{proof}
		Let $\{ x_1, x_2, \ldots, x_m\}$ be a basis of $\mathbb{X},$ where $x_i \in Exp(B_{\mathbb{X}})$. As $S_{\mathbb{Y}}$ has a covering by $q$ number of balls, it follows from [(v), Lemma \ref{lemma:orthogonality}] that there exist $y_1^*, y_2^*, \ldots, y_q^* \in Exp(B_{\mathbb{Y}^*})$ such that $$ \cap_{i =1}^q \{ y \in \mathbb{Y}: y_i^*(y) \leq 0\}= \{0\}.$$  Suppose that for each $1 \leq i \leq q,$ $J(y_i)= \{y_i^*\}.$ Consider the system of $mq$ number of inequalities $ y_j^* (Ax_i) \leq 0,$ where $A \in \mathbb{L}(\mathbb{X}, \mathbb{Y})$ and $1 \leq i \leq m, 1 \leq j \leq q.$  As $ \cap_{i =1}^q \{ y \in \mathbb{Y}: y_i^*(y) \leq 0\}= \{0\},$ $Ax_i=0$ for any $ 1 \leq i \leq m.$ It implies that $A=0.$ Applying Lemma \ref{lemma:operator1},  $S_{\mathbb{L}(\mathbb{X}, \mathbb{Y})}$ admits a covering by $mq$ number of balls.
		
		Recall that $\mathbb{L}(\mathbb{X}, \mathbb{Y})$ is isometrically isometric to $\mathbb{L}(\mathbb{Y}^*, \mathbb{X}^*)$. Proceeding as above we can show that  $\mathbb{L}(\mathbb{Y}^*, \mathbb{X}^*)$ can be covered by $np$ number of balls. Therefore, $S_{\mathbb{L}(\mathbb{X}, \mathbb{Y})}$ can be covered by $k$ number of balls,  where $k = \min \{ mq, np\}.$
\end{proof}

Next, we compute the ball covering number of the space $\mathbb{L}(\mathbb{X}, \mathbb{Y})$ under the assumption that  $\mathbb{X}$ is a finite-dimensional strictly convex and $\mathbb{Y}$ is any finite-dimensional smooth. To do so we need the following lemma.

	\begin{lemma}\cite[Cor. 2.9]{RS}\label{exposed:dual}
	Let $\mathbb{X}, \mathbb{Y}$ be finite-dimensional Banach space. Then 
		\[
	Exp(\mathbb{L}(\mathbb{X}, \mathbb{Y})^*)= \bigg\{ y^* \otimes x: x \in Exp(B_{\mathbb{X}}), y^* \in Exp(B_{\mathbb{Y}^*}) \bigg\},
	\]
	where $[  y^* \otimes x](T)= y^*(Tx).$
\end{lemma}

\begin{theorem}\label{theorem:same number}
	Let $ \mathbb{X}$ be an $m$-dimensional strictly convex Banach space and let $\mathbb{Y}$ be an $n$-dimensional smooth Banach space. Then $S_{\mathbb{L}(\mathbb{X}, \mathbb{Y})}$ has a ball covering consisting of $mn+1$ balls. 
    \end{theorem}

\begin{proof}
	Let $x_1, x_2, \ldots, x_m\in S_{\mathbb{X}}$ be linearly independent and let $y_1^*, y_2^*, \ldots, y_n^* \in S_{\mathbb{Y}^*}$ be  linearly independent. Suppose $x_0 =  - \frac{1}{\|\sum_{i=1}^{m} x_i\|} \sum_{i=1}^m x_i$ and $y_0^* =   \frac{1}{\| \sum_{i=1}^{n} y_i\|} \sum_{i=1}^n y_i^*.$ Since $\mathbb{X}$ is strictly-convex and $\mathbb{Y}$ is smooth, $x_0, x_1, \ldots, x_m \in Exp(B_{\mathbb{X}})$ and $y_0^*, y_1^*, \ldots, y_n^* \in Exp(B_{\mathbb{Y}^*}).$ Clearly, $ \{ y_j^* \otimes x_i: 1 \leq i \leq m, 1 \leq j \leq n \}$ is a linearly independent set in $S_{(\mathbb{L}(\mathbb{X}, \mathbb{Y}))^*}.$ Following Lemma \ref{exposed:dual}, $$\{y_0^* \otimes x_0,  y_j^* \otimes x_i: 1 \leq i \leq m, 1 \leq j \leq n\} \subset  Exp (B_{(\mathbb{L}(\mathbb{X}, \mathbb{Y}))^*}).$$ Observe that $$y_0^* \otimes x_0 =  - \frac{1}{ \|\sum_{i=1}^{m} x_i\| \|\sum_{i=1}^{n} y_i^* \|} \sum_{i=1}^{m} \sum_{j=1}^{n} y_j^* \otimes x_i.  $$
 Therefore, using Theorem \ref{theorem:sufficient condition} we obtain the required results.
	\end{proof}
    
In light of the above result, the following two remarks deserve special attention.

   \begin{remark}
  Theorem \ref{theorem:same number} essentially shows that  $\mathcal{B}_{\min}^{\#}(\mathbb{L}(\mathbb{X},\mathbb{Y})) = mn+1, $ whenever $ \mathbb{X}$ is an $m$-dimensional strictly convex Banach space and $\mathbb{Y}$ is an $n$-dimensional smooth Banach space. Indeed, it follows from \cite{C} that $\mathcal{B}_{\min}^{\#}(\mathbb{L}(\mathbb{X},\mathbb{Y})) \geq mn+1, $ which combined with Theorem \ref{theorem:same number} finishes the proof.
    \end{remark}

    \begin{remark}
        The BCP is clearly an isometric property, which is evident from the fact that Banach spaces over the same field and having the same dimension may very well have different minimum ball covering number. In this context, Theorem \ref{theorem:same number} seems particularly instructive. As an immediate consequence of Theorem \ref{theorem:same number}, we may and do conclude that the isometric aspects governing the minimal ball covering number of $ \mathbb{L}(\mathbb{X},\mathbb{Y}) $ are completely governed by the strict convexity of $ \mathbb{X} $ and the smoothness of $ \mathbb{Y}. $ Furthermore, it is worth remembering that there exist strictly convex Banach spaces $ \mathbb{X}, \widetilde{\mathbb{X}} $ having dimension $ m $ and smooth Banach spaces $ \mathbb{Y}, \widetilde{\mathbb{Y}} $ such that $ \mathbb{L}(\mathbb{X}, \mathbb{Y}) $ is \emph{not} isometrically isomorphic to $ \mathbb{L}(\widetilde{\mathbb{X}}, \widetilde{\mathbb{Y}}). $ Nevertheless, Theorem \ref{theorem:same number} ensures that 
        \[ \mathcal{B}_{\min}^{\#}(\mathbb{L}(\mathbb{X},\mathbb{Y})) = \mathcal{B}_{\min}^{\#}(\mathbb{L}(\widetilde{\mathbb{X}},\widetilde{\mathbb{Y}})), \]
        in all such cases.

    \end{remark}

 So far from our discussion, it is clear that the Birkhoff-James orthogonality plays an essential role in determining the  ball covering number of a finite-dimensional Banach spaces. Motivated by this elementary fact, we characterize the symmetric ball covering number of a finite-dimensional Banach space in terms of Birkhoff–James orthogonality. A ball covering $ \{\mathbf{B_i}\}_{i \in \lambda} $ is called symmetric if $\{\mathbf{B_i}\}_{i \in \lambda} = \{\mathbf{-B_i}\}_{i \in \lambda}.$
  We refer to \cite{C} for further studies in this direction. Beginning with a general compact subset of a Banach space, we will subsequently specialize to the unit sphere. We skip the proof of the next result, as, by arguments similar to Lemma \ref{lemma:orthogonality} it can be easily verified.

\begin{prop}\label{proposition:symmetric ball covering}

    Let $ \mathbb{X} $ be a finite-dimensional Banach space. Let $ \mathcal{A} \subset \mathbb{X} \setminus \{0\} $ be a compact set such that for any non zero $ x \in \mathbb{X}, $ there exists $ r_x >0 $ such that $ r_x x \in \mathcal{A}. $ Then the following are equivalent:
	\begin{itemize}
		\item[$(i)$] $ \mathcal{A} $ admits a symmetric ball covering consisting of $ 2m $ balls,
		
		\item[$(ii)$]  There exist $ x_1, x_2,\ldots, x_m \in \mathbb{X} \setminus \{0\} $ such that $ \cap_{i=1}^m x_i^{\perp} = \{0\}. $
	\end{itemize}
    
\end{prop}

 As a corollary to the above, we now give a necessary and sufficient condition for a symmetric ball covering of the unit sphere of any Banach space.  

\begin{cor}
	Let $ \mathbb{X} $ be a finite-dimensional Banach space. Then the following are equivalent:
	
		\begin{itemize}
		\item[$(i)$] $ S_\mathbb{X} $ admits a symmetric ball covering consisting of $ 2m $ balls,
		
		\item[$(ii)$] There exists $ x_1, x_2,\ldots, x_m \in \mathbb{X} \setminus \{0\} $ such that $ \cap_{i=1}^m x_i^{\perp} = \{0\}. $
	\end{itemize}
\end{cor}

 It is well known that the BCP of a Banach space is not hereditary. Cheng \cite{CCL} showed that $\ell_{\infty} $ has the BCP, but it contains a closed subspace which does not have this property. Keeping that in mind we would like to ask a similar question in the finite-dimensional setting. Observe that if we take $\mathbb{X}=\ell_1^3$ and
\[
\mathbb{Y}:=\{(x,y,0): x,y\in\mathbb{R}\},
\]
then it follows from Corollary~\ref{1 norm} that the ball covering number of $\mathbb{X}$ is $4$. Also, the ball covering number of $\mathbb{Y}$ is equal to $4$. This naturally leads to the following question.

\begin{question}
Does there exists a finite-dimensional Banach space $\mathbb{X}$ and a subspace $\mathbb{Y}$ of $\mathbb{X}$ such that the ball covering number of $\mathbb{Y}$ is strictly greater than the ball covering number of $\mathbb{X}$?
  \end{question}
  
 Although we do not have a complete answer to this question at present, nor any explicit counterexample, we will end our work with the following sufficient condition under which a subspace admits a minimal ball covering with the same number of balls as its ambient space.

\begin{prop}	
	Let $\mathbb{X}$ be a Banach space such that $S_{\mathbb{X}}$ admits a ball covering consisting of $m$ balls. Let  $\mathbb{Y}$ be a subspace of $\mathbb{X}$ such that $\mathbb{Y}$ intersects the interior of every maximal face of $B_{\mathbb{X}}$ containing a smooth point. Then $S_{\mathbb{Y}}$ admits a ball covering consisting of $m$ balls.
\end{prop}

\begin{proof}	
	Since $S_{\mathbb{X}}$ admits a ball covering by $m$ balls, it follows from Lemma \ref{lemma:orthogonality} that there exists a  set $\{x_i: 1 \leq i \leq m \} \subset Sm_\mathbb{X}$ such that $\cap_{i =1}^{m} x_i^- = \{0\}.$
	Suppose that $x_i \in F_i,$ where $F_i$ is a maximal face of $B_{\mathbb{X}}.$ From the hypothesis, $\mathbb{Y} \cap int(F_i) \neq \emptyset. $ Let $y_i \in \mathbb{Y} \cap int(F_i).$ Since $y_i \in int(F_i),$ it is easy to observe that for any $f_{y_i} \in J(y_i),$ $f_{y_i}(z)=1, \forall z \in F_i.$ As $x_i \in F_i, $ $f_{y_i} \in J(x_i).$ 
	Since $F_i \cap Sm(\mathbb{X}) \neq \emptyset$ and $y_i \in int(F_i),$ it is also easy to show that $y_i \in Sm(\mathbb{X}).$ Clearly, $J(y_i)=\{f_{y_i}\}= J(x_i).$ So, $y_i^{-}= x_i^{-}$ and therefore,  $\cap_{i =1}^{m} y_i^- = \{0\}.$ Following Lemma \ref{lemma:orthogonality}, $S_{\mathbb{Y}}$ has a ball covering by $m$ balls.
\end{proof}


\end{document}